\documentclass[english,12pt,a4paper,reqno]{amsart} 

\usepackage{fullpage}
\usepackage{hyperref}
\usepackage{babel}
\usepackage{amssymb}
\usepackage{amsmath}
\usepackage{stmaryrd}
\usepackage{pifont}
\usepackage{textcomp}
\usepackage[mathscr]{euscript}

\newcommand{\CASE}[1]{{\vspace{0.3cm}\noindent\sc #1.}}

\newtheorem{Theorem}{Theorem}[section]

\newtheorem{Proposition}[Theorem]{Proposition}
\newtheorem{Lemma}[Theorem]{Lemma}

\newtheorem{Corollary}[Theorem]{Corollary}
\newtheorem{Example}[Theorem]{Example}

\numberwithin{equation}{section}

\theoremstyle{definition}
\newtheorem{Definition}[Theorem]{Definition}
\newtheorem{Remark}[Theorem]{Remark}
\newtheorem{Setting}[Theorem]{Setting}

\newcommand{\PpC}{{\rm P}p{\rm C}}

\newcommand{\Q}{{\mathbb{Q}}}

\newcommand{\Zge}{{\mathbb{Z}_{\geq 0}}}

\newcommand{\ch}{{\rm char}}
\newcommand{\PSC}{{{\rm P}S{\rm C}}}
\newcommand{\PSCL}{{{\rm P}S{\rm CL}}}
\newcommand{\PSCC}{{{\rm P}S{\rm CC}}}
\newcommand{\PStauCL}{{{\rm P}S^\tau{\rm CL}}}
\newcommand{\PStauCC}{{{\rm P}S^\tau{\rm CC}}}

\newcommand{\PmcSC}{{{\rm P}\mcS{\rm CL}}}
\newcommand{\PmcSCC}{{{\rm P}\mcS{\rm CC}}}
\newcommand{\PCC}{{\rm PCC}}
\newcommand{\PRC}{{\rm PRC}}
\newcommand{\KtotS}{{K_{{\rm tot},S}}}

\newcommand{\notdivides}{\mbox{$\not|$}}
\newcommand{\smallnotdivides}{{\not\;|}}
\newcommand{\mf}{\mathfrak}
\newcommand{\mc}{\mathcal}
\newcommand{\mfp}{{\mf p}}

\newcommand{\mfP}{{\mf P}}
\newcommand{\mfQ}{{\mf Q}}
\newcommand{\mcS}{{\mathcal{S}}}
\newcommand{\mcP}{{\mathcal{P}}}

\newcommand{\mcO}{{\mathcal{O}}}
\newcommand{\mcL}{{\mathcal{L}}}
\newcommand{\mcLr}{{\mcL_{\rm ring}}}

\newcommand{\tp}{{\rm tp}}
\newcommand{\CC}{{\rm CC}}
\newcommand{\mfpinS}{{\mfp\in S}}

\def\hefresh{
   \def\hefreshD{\mathop{\raise1.5pt\hbox{${\smallsetminus}$}}}
   \def\hefreshS{\mathop{\raise0.85pt\hbox{$\scriptstyle\smallsetminus$}}}
   \mathchoice{\hefreshD}{\hefreshD}{\hefreshS}{\hefreshS}}
\renewcommand{\setminus}{\hefresh}

\def\dotcup{
 \def\dotcupD{\mathbin{\mathop{\smash\cup}\limits^\cdot}}
 \def\dotcupS{\mathbin{\mathop{\smash\cup}\limits^\cdot}}
 \mathchoice{\dotcupD}{\dotcupD}{\dotcupS}{}}
\def\dotunion{
\def\dotunionD{\bigcup\kern-10.5pt\cdot\kern5pt}
\def\dotunionT{\bigcup\kern-8.5pt\cdot\kern3.5pt}
\mathop{\mathchoice{\dotunionD}{\dotunionT}{}{}}}
\renewcommand{\biguplus}{\dotunion}
\renewcommand{\uplus}{\dotcup}

\newcommand{\bfalpha}{\mbox{\boldmath$\alpha$}}

\newcommand{\axS}{{\rm(\hyperref[S1]{S1})}}
\newcommand{\axSS}{{\rm(\hyperref[S2]{S2})}}
\newcommand{\axSSS}{{\rm(\hyperref[S3]{S3})}}
\newcommand{\axSSSS}{{\rm(\hyperref[S4]{S4})}}
\newcommand{\axA}{{\rm(\hyperref[A1]{A1})}}
\newcommand{\axAA}{{\rm(\hyperref[A2]{A2})}}
\newcommand{\axAAA}{{\rm(\hyperref[A3]{A3})}}
\newcommand{\axAAAA}{{\rm(\hyperref[A4]{A4})}}
\newcommand{\axAAAAA}{{\rm(\hyperref[A5]{A5})}}
\newcommand{\axB}{{\rm(\hyperref[B1]{B1})}}
\newcommand{\axBB}{{\rm(\hyperref[B2]{B2})}}
\newcommand{\axC}{{\rm(\hyperref[C1]{C1})}}
\newcommand{\axCC}{{\rm(\hyperref[C2]{C2})}}
\newcommand{\axD}{{\rm(\hyperref[D1]{D1})}}
\newcommand{\axDD}{{\rm(\hyperref[D2]{D2})}}
\newcommand{\axE}{{\rm(\hyperref[E1]{E1})}}
\newcommand{\axEE}{{\rm(\hyperref[E2]{E2})}}
\newcommand{\axH}{{\rm(\hyperref[H1]{H1})}}
\newcommand{\axHH}{{\rm(\hyperref[H2]{H2})}}
\newcommand{\axHHH}{{\rm(\hyperref[H3]{H3})}}
\newcommand{\axN}{{\rm(\hyperref[N1]{N1})}}
\newcommand{\axNN}{{\rm(\hyperref[N2]{N2})}}

\begin{document}

\author{Arno Fehm}
\title{Elementary geometric local-global principles\\ for fields}

\begin{abstract}
We define and investigate a family of local-global principles for fields 
involving both orderings and $p$-valuations.
This family contains the PAC, PRC and P$p$C fields and
exhausts the class of pseudo classically closed fields.
We show that the fields satisfying such a local-global principle
form an elementary class, admit diophantine definitions of holomorphy domains,
and their orderings satisfy the strong approximation property.
\end{abstract}

\maketitle

\bibliographystyle{alpha}

\section{Introduction}

\subsection{Geometric local-global principles}

The topic of this work is the study of {\em geometric local-global principles} for fields
from a model theoretic point of view.
Here, a field $F$ is said to satisfy a geometric local-global principle
for a class of $F$-varieties $\mathcal{V}$ 
and a family $\mathcal{F}$ of extensions of $F$ if each
$V\in\mathcal{V}$ has an $F$-rational point if and only if
it has $F'$-rational points for all $F'\in\mathcal{F}$.
For example, the classical Hasse-Minkowski theorem tells us
that a quadric $V$ over $F=\Q$ has a $\Q$-rational point if and only
if it has rational points over each of the completions $\mathbb{R},\Q_2,\Q_3,\dots$ of $\Q$.
However, this does not hold for arbitrary $\Q$-varieties $V$.
We are interested in fields $F$ that satisfy a geometric local-global principle
for {\em all} $F$-varieties.

A well-studied class of such fields consists of Prestel's {\em pseudo real closed} (PRC) fields,
defined by the property that every $F$-variety that has a smooth rational
point over every real closure of $F$ has an $F$-rational point
\cite{PrestelPRC}, \cite{Ershovrclosed2}, \cite{Prestel1983}
-- a prominent example of a field with this property is the field $\Q_{\rm tr}$
of totally real algebraic numbers.
Among other things, it was shown that the class of PRC fields is elementary
in the language of rings.
That is, the PRC property can be formulated in (possibly infinitely many)
sentences of first-order logic.
Similar work was done for the $p$-adic analogue, the $\PpC$ fields
\cite{Grob}, \cite{HaranJarden88}, \cite{Kuenzi}.
Examples of further modifications and generalizations are
\cite{Kuenzi2}, \cite{ErshovPi}, and \cite{DarniereHasse}. 

\subsection{Pseudo classically closed fields}
The aim of this work is to give a common framework 
for several geometric local-global principles that came up in recent years.
For example, let $S$ be a finite set of absolute values on a number field $K$
and let $\KtotS$ denote the maximal Galois extension of $K$ contained in all of the completions
$\hat{K}_\mfp$, $\mfp\in S$ -- the field of {\em totally $S$-adic numbers}.
It was proven that the field $\KtotS$, as well as certain subfields $F$ of $\KtotS$ satisfy a 
geometric local-global principle -- they are
{\em pseudo-$S$ closed} ($\PSC$):
A $K$-variety $V$ that has smooth $\hat{K}_\mfp$-rational points for all
$\mfp\in S$ has $F$-rational points,
\cite{MoretBailly}, \cite{GPR}, \cite{Pop}, \cite{JardenRazon}, \cite{GeyerJarden}.
This notion of $\PSC$ fields has been defined and studied only for algebraic extensions of $K$.

Another class of interest consists of the pseudo classically closed (PCC) fields of \cite{Popclassical}.
The class of PCC fields contains all $\PRC$ and all $\PpC$ fields,
and the notion PCC is defined for arbitrary fields.
Note however, that the class of PCC fields is not elementary.

In this work we define a family of local-global principles for fields of characteristic zero,
all of which are elementary. Both $\PRC$, $\PpC$ and $\PSC$ fields
are special cases, and all PCC fields are covered.

\subsection{Results}

Let $K$ be a number field, $S$ a finite set of orderings and valuations on $K$, 
$\tau=(e,f)\in\mathbb{N}^2$ a pair of positive integers,
and $F$ an extension of $K$.
For $\mfp\in S$ we denote by $\mcS_\mfp^\tau(F)$ the set of all orderings and $p$-valuations
of $F$ extending $\mfp$, where in the case of $p$-valuations we demand in addition that
the relative initial ramification index and residue degree are at most $e$ resp.~$f$.
We say that $F$ is $\PStauCC$ if $F$ satisfies a geometric local-global principle
for all $F$-varieties with respect to the family of real and $p$-adic closures
of $F$ at elements of $\bigcup_{\mfp\in S}\mcS_\mfp^\tau(F)$.

Note the following special cases\footnote{Here, $\infty$ denotes the unique ordering and $p$ denotes the $p$-adic valuation on $\Q$.}:
\begin{enumerate}
\item $S=\emptyset$: $F$ is $\PStauCC$ $\Leftrightarrow$ $F$ is PAC (see e.g. \cite[Chapter 11]{FJ3})
\item $K=\Q$, $S=\{\infty\}$: $F$ is $\PStauCC$ $\Leftrightarrow$ $F$ is PRC
\item $K=\Q$, $S=\{p\}$, $\tau=(1,1)$: $F$ is $\PStauCC$ $\Leftrightarrow$ $F$ is $\PpC$
\item $K=\Q$: $F$ is $\PStauCC$ for some $S$ and $\tau$ $\Leftrightarrow$ $F$ is PCC
\item $\tau=(1,1)$, $F\subseteq\KtotS$: $F$ is $\PStauCC$ $\Leftrightarrow$ $F$ is $\PSC$
\end{enumerate}

In particular, our main results generalize the corresponding results for $\PRC$ and $\PpC$ fields:

\begin{Theorem}\label{thm1}
The class of $\PStauCC$ fields is elementary in the language
$\mcL_{\rm ring}(K)$ of rings with constants from $K$.
\end{Theorem}

The most important ingredient in the proof is the following definability result:

\begin{Theorem}\label{thm2}
If $F$ is $\PStauCC$ and $\mfp\in S$, then the holomorphy domain
$$
 \bigcap_{\mfP\in\mcS_\mfp^\tau(F)}\mcO_\mfP
$$
where $\mcO_\mfP$ is the positive cone resp.~valuation ring of $\mfP$,
is uniformly diophantine in $F$ over $K$.
\end{Theorem}
This means that this holomorphy domain is the projection of the zero set of a polynomial over $K$ which is independent of $F$.

Prestel proved that the orderings of any $\PRC$ fields satisfy
the so-called strong approximation property:
Given an open-closed set of orderings one can find an element which is positive 
at all of those, and negative at all the other orderings.
We show that this result extends to $\PStauCC$ fields:

\begin{Theorem}\label{thm3}
If $F$ is $\PStauCC$, then $\mcS_\mfp^\tau(F)$ satisfies the strong approximation property
for each $\mfp\in S$. 
\end{Theorem}

Extending the notion of totally real field extensions
we call an extension $E/F$ {\em totally $S^\tau$-adic}
if every element of $\mcS_\mfp^\tau(F)$ extends to an element of
$\mcS_\mfp^\tau(E)$ of the same type (i.e.~same residue field and same initial ramification, in the case of $p$-valuations).
Combining Theorems \ref{thm1}, \ref{thm2} and \ref{thm3} we get the following corollary.

\begin{Corollary}\label{cor1}
If $F$ is $\PStauCC$ and $F\prec F^*$ is an elementary extension,
then $F^*/F$ is totally $S^\tau$-adic.
\end{Corollary}

\noindent
These results have immediate consequences for PCC fields:

\begin{Corollary}\label{corPCC}
Let $F$ be a $\PCC$ field.
\begin{enumerate}
\item The intersection over all $p$-valuation rings of $F$ for any $p$,
 as well as the intersection over all positive cones of $F$, are existentially $\emptyset$-definable in $F$.
\item If $E\equiv F$, then $E$ is $\PCC$.
\item If $F\prec E$, then every ordering and every $p$-valuation of $F$ extends to
 an ordering resp.~$p$-valuation of $E$ of the same type.
\item
The space of orderings of $F$ satisfies the strong approximation property.
\end{enumerate}
\end{Corollary}

In fact, we prove everything in greater generality, without the assumption that $K$ is a number field.
The results of this work also answer a question posed by Darni\`ere in \cite{DarnierePAC} and
play a crucial role
in the axiomatization and proof of decidability
of $\KtotS$ and certain subfields of it in \cite{AFKtotSsigma}.

\subsection*{Acknowledgements}

This work extends Chapter 2 of the author's Ph.D. thesis supervised by Moshe Jarden at Tel Aviv University.
As such, it was supported by the European Commission under contract MRTN-CT-2006-035495.
The author would like to thank the Max Planck Institute Bonn for the hospitality during preparation of this manuscript.

I would like to express my sincere gratitude to Moshe Jarden for his valuable guidance and ongoing support.
I also would like to thank Ido Efrat and Dan Haran for helpful discussion concerning Section \ref{sec6},
Itay Kaplan and Alexander Prestel for their contribution to Section \ref{sec:quantifyCC},
and the anonymous referee for many clarifying remarks and questions.

\section{Preliminaries and Notation}
\label{secprelim}

\subsection{Notation}
Every ring and every semiring is commutative with $1$.
If $R$ is a ring, we denote by $R^\times$ the group of invertible elements of $R$.
If $K$ is a field, we denote by $\tilde{K}$ 
a fixed algebraic closure of $K$.
The cardinality of a set $X$ is denoted by $|X|$.
By $\uplus$ we denote the disjoint union of sets.
Varieties are geometrically irreducible and geometrically reduced.
If $V$ is a $K$-variety and $K\subseteq L$ a field extension
we denote by $L(V)$ the function field of $V$ over $L$.

\subsection{Model Theory}
\label{secModelTheory}
For the basic notions of model theory see for example \cite{Marker}.
The language of rings is
$\mcLr=\{+,-,\cdot,0,1\}$ where $+$ and $\cdot$ are binary function symbols,
$-$ is a unary function symbol,
and $0$ and $1$ are constant symbols.
If $\mcL$ is a language containing $\mcLr$,
$K$ is an $\mcL$-structure,
and $C$ is a subset of $K$,
we denote by 
$\mcL(C)=\mcL\cup\{c_x\colon x\in C\}$ 
the language $\mcL$
augmented by constant symbols for the elements in $C$.
If $\varphi(x_1,\dots,x_n)$ is an $\mcL$-formula in $n$ free variables,
and $K$ is an $\mcL$-structure, we denote by 
$\varphi(K)=\{ \mathbf{a}\in K^n\colon K\models\varphi(\mathbf{a}) \}$
the subset defined by $\varphi$ in $K$.

\subsection{Real Closed Fields}
\label{secRCF}
We assume familiarity with the theory of ordered and real closed fields
as presented in \cite{Prestel1}, and only recall a few definitions
and facts.

A {\bf positive cone} of a field $K$ is a semiring $P\subseteq K$ (i.e.~$0,1\in P$, $P+P\subseteq P$, $P\cdot P\subseteq P$)
such that $P\cup(-P)=K$ and $P\cap(-P)=\{0\}$.
A field is {\bf real closed} if it has an ordering but
each proper algebraic extension has no ordering.
A real closed field $K$ has a unique ordering, 
given by the positive cone $K^2$, \cite[3.2]{Prestel1}.
A real closed field $F$ is a {\bf real closure} of an ordered field $K$
if $F$ is an algebraic extension of $K$ and the unique ordering of $F$ extends 
the ordering of $K$.
Any ordered field $K$ has a real closure,
which is unique up to $K$-isomorphism, \cite[3.10]{Prestel1}.

The language of ordered rings
$\mcL_\leq=\mcLr\cup\{\leq\}$ 
is the language of rings augmented by a binary relation symbol $\leq$,
which is interpreted as the ordering of an ordered field.
The $\mcL_{\leq}$-theory of real closed ordered fields is complete and has 
effective quantifier elimination, \cite[3.3.15, 3.3.16]{Marker}.

\subsection{Valued Fields}
\label{secvalued}
We assume familiarity with
the basics of valuation theory, see e.g.~\cite{EnglerPrestel}.

If $v:K\twoheadrightarrow\Gamma\cup\{\infty\}$ is a valuation on a field $K$ with value group $\Gamma$
we denote by $\mcO_v$ the valuation ring,
by $\mathfrak{m}_v$ its maximal ideal,
and by $\bar{K}_v=\mcO_v/\mathfrak{m}_v$ the residue field.
We say that $v$ is of {\bf rank one}
if its value group has no non-trivial proper convex subgroup,
and {\bf discrete} if its value group is discrete in the order topology.
We normalize every discrete valuation such that $\mathbb{Z}$
is a convex subgroup of the value group.
We will use the following variant of Hensel's lemma, \cite[4.1.3(5)]{EnglerPrestel}.

\begin{Lemma}\label{HenselRychlik}
Let $v$ be a Henselian valuation on $K$.
If $f\in\mcO_v[X]$ and $a\in\mcO_v$ with $v(f(a))>2v(f^\prime(a))$,
then there exists $\alpha\in\mcO_v$ with $f(\alpha)=0$
and $v(a-\alpha)>v(f^\prime(a))$.
\end{Lemma}

The language of valued fields
$\mcL_R=\mcLr\cup\{R\}$
is the language of rings
augmented by a unary predicate symbol $R$, which is interpreted as the valuation ring of a valued field.

\subsection{$p$-adically Closed Fields}
\label{secpCF}
We recall the notion of $p$-adically closed fields
and quote some well known results from \cite{PR}.

A valuation $v$ on a field $K$ of characteristic zero
with residue field of characteristic $p>0$
is a {\bf $p$-valuation of $p$-rank $d\in\mathbb{N}$}
if 
${\rm dim}_{\mathbb{F}_p}\mcO_v/p\mcO_v=d$.
The residue field $\bar{K}_v$ of a $p$-valued field $(K,v)$ is finite,
and the value group $v(K^\times)$ is discrete and $v(p)\in\mathbb{Z}$.
If $e=v(p)$ and $f=[\bar{K}_v:\mathbb{F}_p]$, then $d=ef$, \cite[p.~15]{PR}.
We call $(p,e,f)$ the {\bf type} of $(K,v)$.

A $p$-valued field is {\bf $p$-adically closed} if it has no proper $p$-valued algebraic extension
of the same $p$-rank.
Every $p$-adically closed valued field $(K,v)$ has a unique $p$-valuation, \cite[6.15]{PR}.
A {\bf $p$-adic closure} of a $p$-valued field $(K,v)$ is an
algebraic extension of $(K,v)$ which is $p$-adically closed of the same $p$-rank as $(K,v)$.
A $p$-valued field $(K,v)$ is $p$-adically closed if and only if
it is Henselian and the value group $v(K^\times)$ is a $\mathbb{Z}$-group, \cite[3.1]{PR}.
Here, an ordered abelian group $\Gamma$ is a {\bf $\mathbb{Z}$-group}
if it is discrete and $(\Gamma:n\Gamma)=n$ for each $n\in\mathbb{N}$.
Any $p$-valued field $(K,v)$ has a $p$-adic closure.
A $p$-adic closure of $(K,v)$ is unique up to $K$-isomorphism 
if and only if $v(K^\times)$ is a $\mathbb{Z}$-group, \cite[3.2]{PR}.

\section{Classical Primes}
\label{sec1}

\noindent
We start by introducing the notion of a classical prime.
This notion generalizes the notion of a place of a number field
and unifies considerations about orderings and $p$-valuations.

\begin{Definition}
A {\bf prime} $\mf p$ of a field $K$ is
either an equivalence class of valuations on $K$ ($\mf p$ is a {\bf non-archimedean} prime)
or an ordering of $K$ ($\mf p$ is an ${\bf archimedean}$ prime).
If $\mfp$ is an equivalence class of valuations, 
let $v_{\mf p}$ be a fixed valuation in the class $\mf p$,
let $p_\mfp=\ch(\bar{K}_\mfp)$,
the characteristic of the {\bf residue field} 
$\bar{K}_\mfp=\bar{K}_{v_\mfp}$,
and denote by 
$$
 {\mc O}_{\mf p}=\{ x\in K \colon v_{\mf p}(x)\geq 0 \}
$$ 
the corresponding valuation ring.
If $\mfp$ is an ordering, 
denote $\mfp$ by $\leq_{\mf p}$,
let $p_\mfp=\infty$,
and denote by
$$
 {\mc O}_{\mf p}=\{ x\in K \colon x \geq_{\mf p} 0 \}
$$ 
the corresponding positive cone.
The {\bf localization} 
$K_\mfp$
of $K$ with respect to $\mfp$
is a Henselization of $(K,v_\mfp)$ (if $p_\mfp\neq\infty$) 
resp.~a real closure of $(K,\leq_\mfp)$ (if $p_\mfp=\infty$).
It is unique up to $K$-isomorphism.
\end{Definition}

\begin{Remark}\label{complexprimes}
The reader may have noticed that our definition of primes does not include the classical so called 
\textquoteleft complex primes\textquoteright,
i.e.~absolute values for which the corresponding completion is isomorphic to 
$\mathbb{C}$.
The reason for this omission is that both for the ${\rm P}S{\rm C}$ property and for the 
definition of the fields $\KtotS$ we are interested in,
the \textquoteleft complex primes\textquoteright~in $S$ can be disregarded.
\end{Remark}

\begin{Example}\label{Example1}
The field $\mathbb{Q}$ has one archimedean prime, which we denote by $\infty$,
and one non-archimedean prime for each prime number $p$, which we simply denote by $p$.
Note that in our notation, $\mathbb{Q}_p$ is now the field of $p$-adic {\em algebraic} numbers --
we denote the field of $p$-adic numbers by $\hat{\mathbb{Q}}_p$.
\end{Example}

\begin{Definition}\label{liesover}
Let $F/K$ be an extension of fields.
A prime $\mfP$ of $F$ {\bf lies above} a prime $\mfp$ of $K$
if 
$\mathcal{O}_\mfP\cap K=\mathcal{O}_\mfp.$
We write this as $\mfP|_K=\mfp.$
If $\mfp$ is a prime of $K$ and $\sigma\in{\rm Aut}(K)$ is an automorphism of $K$, 
then the {\bf conjugate}
$\sigma\mfp$ 
of $\mfp$ is the unique prime of $K$ with $\mcO_{\sigma\mfp}=\sigma(\mcO_\mfp)$.
\end{Definition}

\begin{Definition}\label{DefprimeCC}\label{defquasilocal}
A {\bf classical} prime $\mfp$ of $K$
is either an equivalence class of $p$-valuations, 
for some prime number $p$,
or an ordering of $K$.
For a classical prime $\mfp$ of $K$, a {\bf classical closure}
of $(K,\mfp)$ is a
$p$-adic closure of $(K,v_\mfp)$ resp.~a real closure of $(K,\leq_\mfp)$.
Let 
$\CC(K,\mfp)$ 
denote the set of all classical closures of $(K,\mfp)$
contained in $\tilde{K}$.
We say that $(K,\mfp)$ is {\bf classically closed} if $K\in\CC(K,\mfp)$,
i.e.~if $K$ is $p$-adically closed resp.~real closed.
A prime $\mfp$ of $K$ is {\bf local}
if it is classical and
%if $\hat{K}_{\mf p}$ is a local field, 
%i.e.~a finite extension of $\hat{\mathbb{Q}}_p$, or $\mathbb{R}$.
%It is {\bf of rank 1}
the value group of $v_\mfp$ is isomorphic to $\mathbb{Z}$
resp.~the ordering $\leq_\mfp$ is archimedean.
%A non-archimedean prime $\mf p$ is {\bf discrete}
%if the value group of $\mf p$ has a smallest positive element.
A classical prime $\mfp$ of $K$ is {\bf quasi-local} if $K_\mfp\in\CC(K,\mfp)$,
i.e.~if the localization is a classical closure.
\end{Definition}

\begin{Remark}\label{Remarkpadic}\label{numberfieldlocal}\label{remarkquasilocal}
Note that 
this definition of local primes essentially coincides with
the definition of local primes in \cite{GeyerJarden} and \cite{HJPe},
and the 
\textquoteleft classical P-adic valuations and orderings\textquoteright~in \cite{HJPd},
except for the complex primes (cf.~Remark~\ref{complexprimes}) .
A non-archimedean classical prime is quasi-local
if and only if its value group is a $\mathbb{Z}$-group,
cf.~Section \ref{secpCF}.
If $\mfp$ is quasi-local, 
then all $K^\prime\in\CC(K,\mfp)$ are $K$-conjugate.
Each prime of a number field is local,
and each local prime is quasi-local.
\end{Remark}

\begin{Definition}
The {\bf type} 
$\tp(\mfp)=(p_\mfp,e_\mfp,f_\mfp)$ 
of a classical prime $\mfp$ of $K$ is
the type $(p,e,f)$ of the $p$-valuation $v_\mfp$ if $p_\mfp=p$,
and $(\infty,1,1)$ if $p_\mfp=\infty$.
If $\mfP$ lies above $\mfp$, 
then the {\bf relative type} of $\mfP$ over $\mfp$ is 
$\tp(\mfP/\mfp)=(e_\mfP/e_\mfp,f_\mfP/f_\mfp)\in\mathbb{N}^2$.
We introduce a partial ordering on the set $\mathbb{N}^2$ 
of relative types by defining
$(e,f)\leq(e',f')$ if $e\leq e'$ and $f|f'$,
and a partial ordering on the set of types by defining
$(p,e,f)\leq(p',e',f')$ if $p=p'$, $e\leq e'$ and $f|f'$.
Since in a classically closed field $(F,\mfP)$ the prime $\mfP$ is unique, \cite[Prop.~7.2(c)]{HJPb},
we write $\mfP_F=\mfP$ and $\tp(F)=\tp(\mfP_F)$.
\end{Definition}

\begin{Definition}\label{DefR}
We say that a field $F$ is $\mathbf{P}\mathcal{F}\mathbf{C}$  with respect
to a family $\mathcal{F}$ of algebraic extensions of $F$
if every smooth $F$-variety $V$
has an $F$-rational point, provided it has an $F^\prime$-rational point for each $F^\prime\in\mathcal{F}$,
cf.~\cite[\S 7]{Jarden1991}.

If $\mcS$ is a set of primes of $F$, then
$F$ is {\bf pseudo-$\mcS$-closed with respect to localizations} ($\mathbf{P\mcS CL}$)
if it is ${\rm P}\mathcal{F}{\rm C}$ with respect to the family 
$$
 \mathcal{F}=\{F_\mfP \colon \mfP\in\mcS\}
$$
of localizations.
If $\mcS$ is a set of classical primes of $F$, then
$F$ is {\bf pseudo-$\mcS$-closed with respect to classical closures} ($\mathbf{P\mcS CC}$)
if it is ${\rm P}\mathcal{F}{\rm C}$ with respect to the family 
$$
 \mathcal{F}=\bigcup_{\mfP\in\mcS}\CC(F,\mfP)
$$
of classical closures.
If $\mcS$ is a set of primes of $F$, then
$$
 R(\mcS)=\bigcap_{\mfP\in\mcS}\mathcal{O}_\mfP
$$
is the {\bf holomorphy domain}\footnote{Note that if $\mcS$ contains archimedean primes, then $R(\mcS)$ is only a semiring but not a ring.} of $\mcS$.
\end{Definition}

\begin{Remark}\label{PSCCimpliesPSCold}
Since every classical closure is Henselian resp.~real closed,
$\PmcSCC$ implies $\PmcSC$.
However, the converse does not hold.
\end{Remark}

\section{$\PStauCC$, $\PStauCL$, and $\PCC$ Fields}

\noindent
In this section we define the class of fields we are working with.
For the rest of this work, we fix the following setting.

\begin{Setting}\label{SettingPSCC}~
\begin{enumerate}
 \item[\textbullet] {\bf $K$ is a fixed base field of characteristic $0$.}
 \item[\textbullet] {\bf $S$ is a finite set of local primes of $K$.}
 \item[\textbullet] {\bf $\tau\in\mathbb{N}^2$ is a relative type.}
 \item[\textbullet] {\bf $F$ is an extension of $K$.}
\end{enumerate}
\end{Setting}

\begin{Definition}\label{DefCC}
For $\mfp\in S$
denote by $\mcS_\mfp^\tau(F)$ 
the set of all classical primes $\mfP$ of $F$ lying above $\mfp$ with
$\tp(\mfP/\mfp)\leq\tau$. Also, let
\begin{eqnarray*}
 \mcS_S^\tau(F)&=&\bigcup_{\mfp\in S}\mcS_\mfp^\tau(F),\\
  R_\mfp^\tau(F)&=&R(\mcS_\mfp^\tau(F)),\\
 {\rm CC}_\mfp^\tau(F)&=&\bigcup_{\mfP\in\mcS_\mfp^\tau(F)}{\rm CC}(F,\mfP),\\
 {\rm CC}_S^\tau(F)&=&\bigcup_{\mfp\in S}{\rm CC}_\mfp^\tau(F).
\end{eqnarray*}
We say that
$F$ is 
{\bf pseudo-$S^\tau$-closed with respect to localizations} ({\bf P}$S^\tau${\bf CL})
resp.~{\bf pseudo-$S^\tau$-closed with respect to classical closures} ({\bf P}$S^\tau${\bf CC}) if $F$ is 
${\rm P}\mcS{\rm CL}$
resp.~${\rm P}\mcS{\rm CC}$
with respect to
$\mcS=\mcS_S^\tau(F)$.
\end{Definition}

\begin{Remark}\label{PSCCimpliesPSC}\label{PCM}
Note that $F$ is $\PStauCC$ if and only if it is ${\rm P}\mathcal{F}{\rm C}$
with respect to the family 
$\mathcal{F}=\CC_S^\tau(F)$.
If $F$ is $\PStauCC$, then $F$ is $\PStauCL$,
cf.~Remark~\ref{PSCCimpliesPSCold}.
In the case $\tau=(1,1)$ we will drop $\tau$ in all notations, 
and write for example $\PSCC$ instead of $\PStauCC$.

Note that for $K=\mathbb{Q}$ and $|S|=1$,
our notion of $\PSCC$ fields coincides with the classical notions
of $\PpC$ resp.~$\PRC$ fields.
For $K=\mathbb{Q}$ and $S$ a finite set of prime numbers (cf.~Example \ref{Example1}), 
the notion of $\PSCC$ fields coincides with the notion
of ${\rm PC}_M$ fields of \cite{Kuenzi2} and \cite{Kuenzi3}.
For $K=\Q$ and $S=\emptyset$, a $\PSCC$ field
is just a PAC field, cf.~\cite[Chapter~11]{FJ3}.

Note that there is a related notion of 
${\rm P}S{\rm C}$ fields in the literature (cf.~the Introduction).
However, in 
\cite{JardenRazon} and \cite{GeyerJarden}
this property is defined only for algebraic extensions of $K$,
and in 
\cite{JardenRazonGeyer}, \cite{RazonDensity} and \cite{HJPe}
only for subextensions of $\KtotS/K$.
For subextensions of $\KtotS/K$, the three notions 
$\PSC$, $\PSCL$, and $\PSCC$ coincide,
but both the $\PSCL$ property and the $\PSCC$ property
are defined for arbitrary extensions of $K$.
The reason for our focus on the $\PSCC$ property is that, 
as we show, it is elementary.
\end{Remark}

We now briefly recall Pop's definition of a pseudo classically closed field
and show how it fits into the picture.

\begin{Definition}
Let $\CC(F)$ denote the set of all classical closures of $F$ with respect to arbitrary classical primes of $F$.
A {\bf classical field} is either $\mathbb{R}$ or a finite extension of $\hat{\Q}_p$ for some $p$.
If $E$ is a classical field, let
${\rm loc}^E(F)$ be the set of all algebraic extensions of $F$ that
are $\mcL_{\rm ring}$-elementarily equivalent to $E$. 
A field $F$ is $\mathbf{PCC}$ if there exists a finite family of classical fields $\mathcal{E}$ 
such that $F$ is ${\rm P}\mathcal{F}{\rm C}$ for $\mathcal{F}=\bigcup_{E\in\mathcal{E}}{\rm loc}^E(F)$.\footnote{This
definition coincides with the original one 
since the ${\rm P}\mathcal{F}{\rm C}$ property is preserved under enlarging $\mathcal{F}$.}
\end{Definition}

\begin{Lemma}\label{Fmin}
If $F$ is ${\rm P}\mathcal{F}{\rm C}$
with respect to $\mathcal{F}=\bigcup_{E\in\mathcal{E}}{\rm loc}^E(F)$ for a finite family 
of classical fields $\mathcal{E}$,
then $F$ is also ${\rm P}\mathcal{F}_{\rm min}{\rm C}$,
where $\mathcal{F}_{\rm min}$ is the set of minimal elements of $\mathcal{F}$.
Moreover, $\mathcal{F}_{\rm min}=\CC(F)$.
\end{Lemma}

\begin{proof}
See \cite[Theorem 2.3 and Corollary 2.11]{Popclassical}.
\end{proof}

\begin{Proposition}\label{PCCPSCC}
A field is $\PCC$ if and only if it is
$\PStauCC$ for some finite set $S$ of primes of $K=\Q$ and some $\tau\in\mathbb{N}^2$.
\end{Proposition}

\begin{proof}
Suppose $F$ is ${\rm P}\mathcal{F}{\rm C}$
with respect to $\mathcal{F}=\bigcup_{E\in\mathcal{E}}{\rm loc}^E(F)$ for a finite family 
of classical fields $\mathcal{E}=\{E_1,\dots,E_n\}$.
If $\tp(E_i)=(p_i,e_i,f_i)$, let $S=\{p_1,\dots,p_n\}$, $e=e_1\cdots e_n$, $f=f_1\cdots f_n$, and $\tau=(e,f)$.
By Lemma \ref{Fmin} it follows that
$F$ is ${\rm P}\mathcal{F}_{\rm min}{\rm C}$,
and that $\mathcal{F}_{\rm min}=\CC(F)$.
But then $\CC(F)=\CC_S^\tau(F)$,
since if $F'\equiv E_i$, then $F'$ is classically closed and $\tp(F')=\tp(E_i)$, \cite[Prop.~7.2(h)]{HJPb}.
Thus, $F$ is $\PStauCC$.

Conversely, let $F$ be $\PStauCC$. 
Since every classically closed field is elementarily equivalent to a classical field,
and only finitely many types occur among $\CC_S^\tau(F)$,
each of which is the type of only finitely many classical fields,
\cite[Prop.~7.2(j),(k)]{HJPb},
there exists a finite family of classical fields $\mathcal{E}$ such that
$\CC_S^\tau(F)\subseteq\mathcal{F}:=\bigcup_{E\in\mathcal{E}}{\rm loc}^E(F)$.
It follows that $F$ is ${\rm P}\mathcal{F}{\rm C}$, and hence $\PCC$.
\end{proof}

\begin{Definition}
We say that $F$ is {\bf $S^\tau$-quasi-local}
if every $\mfP\in\mcS_S^\tau(F)$ is quasi-local
(cf.~Definition \ref{defquasilocal}).
\end{Definition}

\begin{Lemma}\label{quasilocal}
If $F/K$ is algebraic, then $F$ is $S^\tau$-quasi-local.
\end{Lemma}

\begin{proof}
This follows from the assumption that $S$ consists of local primes:
If $\mfp\in S$ is a $p$-valuation with value group $\mathbb{Z}$
and $\mfP\in\mcS_\mfp^\tau(F)$, then the value group of $\mfP$ is discrete and contained in the divisible hull of $\mathbb{Z}$ since $F/K$ is algebraic,
hence it is isomorphic to $\mathbb{Z}$ itself.
\end{proof}

\begin{Proposition}\label{dense}
If $F$ is $\PStauCC$, then $F$ is $S^\tau$-quasi-local.
\end{Proposition}

\begin{proof}
By Proposition \ref{PCCPSCC}, 
$F$ is ${\rm P}\mathcal{F}{\rm C}$
with respect to $\mathcal{F}=\bigcup_{E\in\mathcal{E}}{\rm loc}^E(F)$ for a finite family 
of classical fields $\mathcal{E}$.
By Lemma \ref{Fmin}, $F$ is ${\rm P}\mathcal{F}_{\rm min}{\rm C}$,
and $\mathcal{F}_{\rm min}=\CC(F)$, hence $\CC_S^\tau(F)\subseteq\mathcal{F}_{\rm min}$.
Thus the claim follows from
\cite[Theorem 2.5 and Theorem 2.9]{Popclassical}.
\end{proof}

\section{Defining Holomorphy Domains}
\label{sec2}

\noindent
This section contains the technical first-order definition
of the holomorphy domains.
For a moment we forget about $K$ and $S$ and consider the following setting.

\begin{Setting}\label{settingS}~
\begin{enumerate}
 \item[\textbullet] $F$ is a field of characteristic zero.
 \item[\textbullet] $\mcS$ is a set of classical\footnote{This condition can be weakened. For example, 
   most results of this section apply also to valuations of residue characteristic zero.}
                   primes of $F$.
 \item[\textbullet] $\mcS$ is partitioned as $\mcS=\biguplus_{i=1}^n \mcS_i$.
 \item[\textbullet] For each $i$, $(p_i',e_i',f_i')\leq(p_i,e_i,f_i)$ are types
                    such that $p_\mfP=p_i$ and $f_\mfP|f_i$ for each $\mfP\in\mcS_i$.
 \item[\textbullet] For each $i$, $\pi_i$ is an element of $F^\times$ that satisfies the following conditions:
\begin{enumerate}
 \item[(S1)\label{S1}] If $\mfP\in\mcS_i$ and $p_i\neq\infty$, then $v_\mfP(\pi_i)>0$ and $v_\mfP(\pi_i)\leq e_i$.
 \item[(S2)\label{S2}] If $\mfP\in\mcS\setminus\mcS_i$ and $p_\mfP\neq\infty$, then $v_\mfP(\pi_i-1)>0$.
 \item[(S3)\label{S3}] If $\mfP\in\mcS_i$ and $p_i=\infty$, then $\pi_i<_\mfP-1$.
 \item[(S4)\label{S4}] If $\mfP\in\mcS\setminus\mcS_i$ and $p_\mfP=\infty$, then $\pi_i>_\mfP 0$.
\end{enumerate}
\end{enumerate}
\end{Setting}

\begin{Definition}
Let $\pi=\prod_{i=1}^n\pi_i$ and
$$
 \mcS_i' = \left\{ \mfP\in\mcS_i : v_\mfP(\pi_i)\leq e_i', f_\mfP|f_i'\right\}
$$ 
if $p_i\neq\infty$, and $\mcS_i'=\mcS_i$ if $p_i=\infty$.
\end{Definition}

Our first goal is to give a first-order definition of the 
holomorphy domain $R(\mcS_i')$ in the case that $F$ is $\PmcSC$.
The case $n=m=1$ of the following lemma can be found in
\cite{HeinemannPrestel}.

\begin{Lemma}\label{absirred}
Let $f\in F[X_1,\dots,X_n]$ and $g\in F[Y_1,\dots,Y_m]$ be non-constant polynomials,
and let $c\in F^\times$.
If $g$ is square-free in $\tilde{F}[\mathbf{Y}]$, then
$$
 h(\mathbf{X},\mathbf{Y}) = f(\mathbf{X})g(\mathbf{Y}) + c\;\;\in F[\mathbf{X},\mathbf{Y}]
$$
is absolutely irreducible.
\end{Lemma}

\begin{proof}
Without loss of generality assume that $F=\tilde{F}$.
We prove the claim by induction on $n$.

First assume that $n=1$.
Let $r({\mathbf Y})$ be any prime factor of $g({\mathbf Y})$.
Since $g$ is square-free, $r|g$ but $r^2\notdivides g$.
Write $h$ as a polynomial in $X_1$. 
Then $r$ divides all coefficients of $h$ except the constant one.
Thus, by Eisenstein's criterion,
$h$ is irreducible in $F({\mathbf Y})[\mathbf{X}]$.
Since $c\neq 0$, it follows that $h$ is irreducible in $F[\mathbf{X},\mathbf{Y}]$.

Now assume that $n>1$ and $f\notin F[X_1]$.
Suppose that $h$ decomposes as $h=h_1h_2$ with $h_1,h_2\in F[\mathbf{X},\mathbf{Y}]\setminus F$.
Since $c\neq 0$ we have $h_1,h_2\notin F[X_1]$:
Indeed, if, say, $h_1\in F[X_1]$, then looking at the constant term of $h$ with respect to $\mathbf{Y}$ gives $h_1(X_1)|f(\mathbf{X})g(0)+c$,
while each non-constant term gives $h_1(X_1)|f(\mathbf{X})$, a contradiction.
Hence, there exists $x\in F$ such that
$h_1(x,X_2,\dots,X_n,\mathbf{Y})\notin F$, $h_2(x,X_2,\dots,X_n,\mathbf{Y})\notin F$, 
and  $f(x,X_2,\dots,X_n)\notin F$.
Consequently, 
$f(x,X_2,\dots,X_n)g(\mathbf{Y})+c$
decomposes in $F[X_2,\dots,X_n,\mathbf{Y}]$, contradicting the induction hypothesis.
\end{proof}

\begin{Lemma}\label{nonsingular}
Let $f\in F[X_1,\dots,X_n]$ be non-constant,
and let $g\in F[Y]$ be non-constant and square-free in $\tilde{F}[Y]$ 
with $g(1)\neq 0$ and $g^\prime(1)\neq 0$.
Then the polynomial
$$
 G(\mathbf{X},Y)=g(Y)(1+f(\mathbf{X}))-g(1) \;\;\in F[\mathbf{X},Y]
$$
is absolutely irreducible, and for every root $\mathbf{x}$ of $f$,
$(\mathbf{x},1)$ is a non-singular point on the hypersurface defined by $G$.
\end{Lemma}

\begin{proof}
This follows from Lemma~\ref{absirred} and a direct computation.
\end{proof}

Our formula defining $R(\mcS_i')$ makes use
of a polynomial of the form $G(\mathbf{X},Y)$ in Lemma~\ref{nonsingular}.
More precisely, 
we let $f(\mathbf{X})$ depend on a parameter $a\in F$
such that $R(\mcS_i')$ consists of all $a\in F$ for which $G(\mathbf{X},Y)$ has a zero in $F$.
We construct $f(\mathbf{X})$ as a product of several polynomials,
each of which has a zero in a certain class of localizations of $F$,
so that the hypersurface $G=0$ has a smooth point in every localization.
The basic idea for this approach appears in \cite{Kuenzi2}.

\begin{Lemma}\label{LemmaA}
Under Setting~\ref{settingS}, the polynomial
$$
 A_i(X)=X^{2^{e_i}}-\pi_i
$$
satisfies the following conditions:
\begin{enumerate}
 \item[(A1)\label{A1}] If $\mfP\in\mcS\setminus\mcS_i$ and $p_\mfP\neq 2$, 
						then $A_i$ has a zero in $F_\mfP$.
 \item[(A2)\label{A2}] If $\mfP\in\mcS_i$ and $p_i\neq\infty$, 
						then for all $x\in F$, $v_\mfP(A_i(x))\leq e_i$.
 \item[(A3)\label{A3}] If $\mfP\in\mcS_i$ and $p_i\neq\infty$, 
						then $v_\mfP(A_i(1))=0$.
 \item[(A4)\label{A4}] $A_i(X)$ is square-free in $\tilde{F}[X]$, and $A_i^\prime(1)\neq 0$.
 \item[(A5)\label{A5}] If $\mfP\in\mcS_i$ and $p_i=\infty$, 
						then for all $x\in F$, $A_i(x)>_\mfP 1$.
\end{enumerate}
\end{Lemma}

\begin{proof}
(A1): If $p_\mfP\neq\infty$, then (A1) follows from $\axSS$ and Hensel's lemma,
otherwise it follows from $\axSSSS$ and the fact that $F_\mfP$ is real closed. 
(A2): The inequality $e_i<2^{e_i}$ implies
that $v_\mfP(x^{2^{e_i}})\neq v_\mfP(\pi_i)\leq e_i$ by $\axS$.
(A3) follows from $\axS$, (A4) from $\ch(F)=0$, and (A5) from $\axSSS$.
\end{proof}

\begin{Lemma}\label{LemmaB}
Under Setting~\ref{settingS}, the polynomial
$$
 B_i(X)=X^{2^{e_i}}+\pi_iX+\pi
$$
satisfies the following conditions:
\begin{enumerate}
 \item[(B1)\label{B1}] If $\mfP\in\mcS\setminus\mcS_i$ and $p_\mfP= 2$, then $B_i$ has a zero in $F_\mfP$.
 \item[(B2)\label{B2}] If $\mfP\in\mcS_i$ and $p_i\neq\infty$, then for all $x\in F$, $v_\mfP(B_i(x))\leq e_i$.
\end{enumerate}
\end{Lemma}

\begin{proof}
(B1) follows from Hensel's lemma.
(B2): The inequality $e_i<2^{e_i}$ implies that 
$v_\mfP(B_i(x))=v_\mfP(\pi)\leq e_i$ if $v_\mfP(x)>0$, 
and $v_\mfP(B_i(x))=v_\mfP(x^{2^{e_i}})\leq 0$ if $v_\mfP(x)\leq 0$.
\end{proof}

\begin{Lemma}\label{LemmaD}
Under Setting~\ref{settingS}, if $p_i\neq\infty$, then for every $a\in F$ the polynomial
$$
 D_{i,a}(X)=a\pi_iX^{2^{e_i}}-X+a
$$
satisfies the following conditions:
\begin{enumerate}
 \item[(D1)\label{D1}] If $\mfP\in\mcS_i$ and $v_\mfP(a)\geq 0$, then $D_{i,a}$ has a zero in $F_\mfP$.
 \item[(D2)\label{D2}] If $\mfP\in\mcS_i$ and $v_\mfP(a)< 0$, then $v_\mfP(D_{i,a}(x))\leq v_\mfP(a)$ for all $x\in F$.
					  Thus, if $v_\mfP(D_{i,a}(x))\geq 0$ for some $x\in F$, then $v_\mfP(a)\geq 0$.
\end{enumerate}
\end{Lemma}

\begin{proof}
(D1) follows from Hensel's lemma and $\axS$.
(D2): If $v_\mfP(x)\geq 0$, then
$v_\mfP(a\pi_ix^{2^{e_i}})>v_\mfP(a)$ and $v_\mfP(x)>v_\mfP(a)$,
so $v_\mfP(D_{i,a}(x))=v_\mfP(a)$.
If $v_\mfP(x)<0$, then the inequality $2^{e_i}\geq e_i+1$ implies that
$v_\mfP(a\pi_ix^{2^{e_i}})<2^{e_i}v_\mfP(x)+v_\mfP(\pi_i)\leq -e_i+v_\mfP(x)+v_\mfP(\pi_i)\leq v_\mfP(x)$
and 
$v_\mfP(a\pi_ix^{2^{e_i}})=v_\mfP(a)+v_\mfP(\pi_i)+2^{e_i}v_\mfP(x)\leq v_\mfP(a)+e_i-2^{e_i}<v_\mfP(a)$,
so $v_\mfP(D_{i,a}(x))=v_\mfP(a\pi_ix^{2^{e_i}})<v_\mfP(a)$.
\end{proof}

\begin{Lemma}\label{LemmaR}
Under Setting~\ref{settingS} and $p_i\neq\infty$, let $d\leq e_i$.
Then the polynomial
$$
 R_{i,d}(X,Y)=(X^{2d}+\pi_i^2)Y^{2^{e_i}}-X^dY+\pi_i^{-1}X^{2d}+\pi_i
$$
satisfies the following condition:
\begin{enumerate}
 \item[(R1)] If $\mfP\in\mcS_i$ with $d|v_\mfP(\pi_i)$, then $R_{i,d}$ has a zero in $F_\mfP$.
 \item[(R2)] If $\mfP\in\mcS_i$ with $d\notdivides v_\mfP(\pi_i)$, then $v_\mfP(R_{i,d}(x,y))\leq {e_i}$ for all $x,y\in F$.
\end{enumerate}
\end{Lemma}

\begin{proof}
First note that with $\gamma(X)= \pi_i^{-1}X^d + \pi_iX^{-d}$
we have $R_{i,d}(X,Y)=X^d D_{i,\gamma(X)}(Y)$.
Furthermore, note that for $x\in F^\times$ and $\mfP\in\mcS_i$,
$v_\mfP(\gamma(x))\geq 0$ if and only if $dv_\mfP(x)=v_\mfP(\pi_i)$.

(R1): 
There exists $x\in F^\times$ with $dv_\mfP(x)=v_\mfP(\pi_i)$, i.e.~$v_\mfP(\gamma(x))\geq 0$.
Therefore, by (D1), $D_{i,\gamma(x)}(Y)$ has a zero $y\in F_\mfP$,
so $(x,y)$ is a zero of $R_{i,d}$.

(R2):
Since $v_\mfP(\gamma(x))<0$ for all $x\in F^\times$,
we have that $v_\mfP(D_{i,\gamma(x)}(y))\leq v_\mfP(\gamma(x))<0$ for all $x\in F^\times$, $y\in F$ by (D2).
Assume that there are $x,y\in F$ with $v_\mfP(R_{i,d}(x,y))>{e_i}$.
Then $x\neq 0$, since $R_{i,d}(0,y)=\pi_i^2(y^{2^{e_i}}+\pi_i^{-1})$, and thus
$v_\mfP(R_{i,d}(0,y))\leq v_\mfP(\pi_i)\leq {e_i}$.
It follows that $v_\mfP(D_{i,\gamma(x)}(y))>v_\mfP(x^{-d})+v_\mfP(\pi_i)\geq v_\mfP(\gamma(x))$,
a contradiction.
\end{proof}

\begin{Lemma}\label{LemmaI}
Under Setting~\ref{settingS} and $p_i\neq\infty$, let $d|f_i$.
Let 
$$
 I_{i,d}(X)=\Phi_{p_i^d-1}(X)\in\mathbb{Z}[X]
$$
be the $(p_i^d-1)$-th cyclotomic polynomial.
Then $I_{i,d}(X)$ satisfies the following conditions.
\begin{enumerate}
 \item[(I1)] If $\mfP\in\mcS_i$ with $d|f_\mfP$,
                 then $I_{i,d}$ has a zero in $F_\mfP$.
 \item[(I2)] If $\mfP\in\mcS_i$ with $d\notdivides f_\mfP$,
               then $v_\mfP(I_{i,d}(x))\leq 0$ for all $x\in F$.
\end{enumerate}
\end{Lemma}

\begin{proof}
Note that $I_{i,d}$ has a zero in the finite field $\bar{F}_\mfP$ if and only if
$\mathbb{F}_{p_i^d}\subseteq\bar{F}_\mfP$, \cite[V \S11 Lemma 3]{BourbakiA2},
which is the case if and only if $d|f_\mfP$.
(I1) follows from Hensel's lemma.
(I2) follows immediately since $I_{i,d}$ is monic.
\end{proof}

\begin{Lemma}\label{LemmaN}
Under Setting~\ref{settingS} and $p_i\neq\infty$,
the polynomial
$$
 N_i(X,Y)=\prod_{d\leq e_i, d\not\leq e_i'}R_{i,d}(X,Y)\cdot\prod_{d|f_i, d\smallnotdivides f_i^\prime}I_{i,d}(X)
$$
satisfies the following conditions:
\begin{enumerate}
 \item[(N1)\label{N1}] If $\mfP\in\mcS_i\setminus\mcS_i^\prime$, then $N_i$ has a zero in $F_\mfP$.
 \item[(N2)\label{N2}] If $\mfP\in\mcS_i^\prime$, then $v_\mfP(N_i(x,y))\leq e_i^2$ for all $x,y\in F$.
\end{enumerate}
\end{Lemma}

\begin{proof}
(N1):
If $v_\mfP(\pi_i)\not\leq e_i^\prime$, then
$R_{i,v_\mfP(\pi_i)}$ has a zero in $F_\mfP$ by (R1).
If $f_\mfP\notdivides f_i^\prime$, then
$I_{i,f_\mfP}$ has a zero in $F_\mfP$ by (I1).

(N2):
Since $v_\mfP(\pi_i)\leq e_i^\prime$ and
$f_\mfP|f_i^\prime$,
it follows that
for all $d\leq e_i$ with $d\not\leq e_i^\prime$,
we have that $d\notdivides v_\mfP(\pi_i)$, and
for $d|f_i$ with $d\notdivides f_i^\prime$
we have
$d\notdivides f_\mfP$.
Therefore, by (R2) and (I2),
$v_\mfP(N_i(x,y))\leq (e_i-e_i')e_i\leq e_i^2$ for all $x,y\in F$.
\end{proof}

\begin{Lemma}\label{MainLemma}
Under Setting~\ref{settingS} and $p_i\neq\infty$, 
let $a\in F$.
If $A_i$ satisfies \axA-\axAAAA,
$B_i$ satisfies \axB-\axBB,
$D_{i,a}$ satisfies \axD-\axDD,
and $N_i$ satisfies \axN-\axNN,
then the polynomial
\begin{eqnarray*}
 G_{i,a}(X,Y,Z)=A_i(Z)(1+\pi_i^{-4e_i-e_i^2}A_i(X)B_i(X)D_{i,a}(X)N_i(X,Y))-A_i(1)
\end{eqnarray*}
satisfies the following conditions:
\begin{enumerate}
 \item If $G_{i,a}$ has a zero in $F$, then $v_\mfP(a)\geq 0$ for all $\mfP\in\mcS_i'$.
 \item If $F$ is $\PmcSC$ and $v_\mfP(a)\geq 0$ for all $\mfP\in\mcS_i'$, then $G_{i,a}$ has a zero in $F$.
\end{enumerate}
\end{Lemma}

\begin{proof}
Let $x,y,z\in F$ with $G_{i,a}(x,y,z)=0$ and let $\mfP\in\mcS_i'$.
Then 
\begin{eqnarray*}
 v_\mfP(1+\pi_i^{-4e_i-e_i^2}A_i(x)B_i(x)D_{i,a}(x)N_i(x,y)) = \\
   \quad\quad\quad = v_\mfP(A_i(1))-v_\mfP(A_i(z)) 
   \geq -e_i
\end{eqnarray*}
by $\axAA$ and $\axAAA$.
Thus, 
$$
 v_\mfP(\pi_i^{-4e_i-e_i^2}A_i(x)B_i(x)D_{i,a}(x)N_i(x,y))\geq -e_i,
$$
so $v_\mfP(D_{i,a}(x))\geq0$
by $\axS$, $\axAA$, and $\axBB$, and $\axNN$.
Therefore, $v_\mfP(a)\geq 0$ by $\axDD$.

Now assume that $F$ is $\PmcSC$ and $v_\mfP(a)\geq 0$ for all $\mfP\in\mcS_i'$.
If $A_i(1)=0$, then $G_{i,a}(0,0,1)=0$. Hence, assume without loss of
generality that $A_i(1)\neq 0$.
Let $\mfP\in\mcS$. We claim that $A_i(X)B_i(X)D_{i,a}(X)N_i(X,Y)$ has a zero in $F_\mfP$.
If $\mfP\in\mcS\setminus\mcS_i$ and $p_\mfP\neq 2$, this follows from $\axA$.
If $\mfP\in\mcS\setminus\mcS_i$ and $p_\mfP= 2$, this follows from $\axB$.
If $\mfP\in\mcS_i\setminus\mcS_i'$, this follows from $\axN$.
If $\mfP\in\mcS_i'$, this follows from $\axD$.
Therefore, by Lemma~\ref{nonsingular} and $\axAAAA$,
$G_{i,a}$ is absolutely irreducible and has a simple zero in $F_\mfP$ for all $\mfP\in\mcS$.
Since $F$ is $\PmcSC$, $G_{i,a}$ has a zero in $F$.
\end{proof}

This almost concludes the proof of the definability of $R(\mcS_i')$ for $p_i\neq\infty$.
We now turn to the case $p_i=\infty$.

\begin{Lemma}\label{LemmaC}
Under Setting~\ref{settingS}, if $p_i=\infty$, then the polynomial
$$
 C(X)=X^2+X+2
$$
satisfies the following conditions:
\begin{enumerate}
 \item[(C1)\label{C1}] If $\mfP\in\mcS\setminus\mcS_i$ and $p_\mfP=2$, then $C$ has a zero in $F_\mfP$.
 \item[(C2)\label{C2}] If $\mfP\in\mcS_i$, then $C(x)>_\mfP 1$ for every $x\in F$.
\end{enumerate}
\end{Lemma}

\begin{proof}
(C1) follows from Hensel's lemma. (C2) is clear.
\end{proof}

\begin{Lemma}\label{LemmaE}
Under Setting~\ref{settingS}, if $p_i=\infty$, then for every $a\in F$, the polynomial
$$
 E_{a}(X)=X^2-a
$$
satisfies the following conditions:
\begin{enumerate}
 \item[(E1)\label{E1}] If $\mfP\in\mcS_i$ and $a\geq_\mfP 0$, then $E_{a}$ has a zero in $F_\mfP$.
 \item[(E2)\label{E2}] If $\mfP\in\mcS_i$, $x,\epsilon\in F$, and $E_{a}(x)\leq_\mfP\epsilon$,
         then $a\geq_\mfP-\epsilon$.
\end{enumerate}
\end{Lemma}

\begin{proof}
(E1) holds since $F_\mfP$ is real closed.
(E2) is obvious.
\end{proof}

\begin{Lemma}\label{LemmaH}
Under Setting~\ref{settingS}, if $p_i=\infty$, then for every $u\in F^\times$, the polynomial
$$
 H_{u}(X)=X^2+u^2
$$
satisfies the following conditions:
\begin{enumerate}
 \item[(H1)\label{H1}] If $\mfP\in\mcS_i$, then for all $x\in F$, $H_u(x)\geq_\mfP u^2$.
 \item[(H2)\label{H2}] If $\mfP\in\mcS_i$, then $H_u(1) = 1+u^2>_\mfP0$.
 \item[(H3)\label{H3}] $H_u(X)$ is square-free in $\tilde{F}[X]$, and $H_u^\prime(1)\neq 0$.
\end{enumerate}
\end{Lemma}

\begin{proof}
All claims are easily verified.
\end{proof}

\begin{Lemma}\label{MainLemmaInfty}
Under Setting~\ref{settingS} and $p_i=\infty$, 
let $a\in F$ and $u\in F^\times$.
If $A_i$ satisfies $\axA$ and $\axAAAAA$,
$C$ satisfies $\axC$-$\axCC$,
$E_a$ satisfies $\axE$-$\axEE$,
and $H_u$ satisfies $\axH$-$\axHHH$,
then the polynomial
$$
 G_{i,a,u}(X,Y)=H_u(Y)(1+A_i(X)C(X)E_{a}(X))-H_u(1)
$$
satisfies the following conditions:
\begin{enumerate}
 \item If $G_{i,a,u}$ has a zero in $F$, then $a\geq_\mfP-{u^{-2}}$ for all $\mfP\in\mcS_i$.
 \item If $F$ is $\PmcSC$ and $a\geq_\mfP 0$ for all $\mfP\in\mcS_i$, then $G_{i,a,u}$ has a zero in $F$.
\end{enumerate}
\end{Lemma}

\begin{proof}
Let $x,y\in F$ such that $G_{i,a,u}(x,y)=0$ and let $\mfP\in\mcS_i$.
Then 
$$
 1+A_i(x)C(x)E_{a}(x)=\frac{H_u(1)}{H_u(y)}\leq_\mfP\frac{1+u^2}{u^2}=1+u^{-2}
$$
by 
$\axH$, $\axHH$.
Thus, $E_{a}(x)\leq_\mfP u^{-2}$ by 
$\axAAAAA$ and $\axCC$.
Therefore, $a\geq_\mfP-u^{-2}$ by $\axEE$.

Now assume that $F$ is $\PmcSC$ and $a\geq_\mfP 0$ for all $\mfP\in\mcS_i$.
If $H_u(1)=0$, then $G_{i,a,u}(0,1)=0$.
Hence, assume without loss of generality that $H_u(1)\neq0$.
Let $\mfP\in\mcS$. We claim that $A_i(X)C(X)E_{a}(X)$ has a zero in $F_\mfP$.
If $\mfP\in\mcS\setminus\mcS_i$ and $p_\mfP\neq 2$, this follows from $\axA$.
If $p_\mfP=2$, it follows from $\axC$.
If $\mfP\in\mcS_i$, it follows from $\axE$. 
Therefore, by Lemma~\ref{nonsingular}, $\axHHH$, and the assumption that $F$ is $\PmcSC$,
it follows that $G_{i,a,u}$ has a zero in $F$.
\end{proof}

For the following proposition, let $A_i$, $B_i$, $C$, $D_{i,a}$, $E_{a}$, $H_u$, $N_i$ be the concrete polynomials defined above.

\begin{Proposition}\label{MainTheoremS}
Under Setting~\ref{settingS}, 
for $p_i\neq\infty$ let $\varphi_i(a)$ be the 
$\mcLr(\pi_1,\dots,\pi_n)$-formula
$$
 (\exists x,y,z)(A_i(z)(1+\pi_i^{-4e_i-e_i^2}A_i(x)B_i(x)D_{i,a}(x)N_i(x,y))-A_i(1)=0),
$$
and for $p_i=\infty$ let $\varphi_i(a)$ be the 
$\mcLr(\pi_i)$-formula
$$
 (\exists u\neq0)(\exists x,y)(a(H_u(y)(1+A_i(x)C(x)E_{a-u^{-2}}(x))-H_u(1))=0).
$$
Then the following holds for the subset $\varphi_i(F)\subseteq F$ defined by $\varphi_i$:
\begin{enumerate}
 \item $\varphi_i(F)\subseteq R(\mcS_i')$.
 \item If $F$ is $\PmcSC$, then $\varphi_i(F)=R(\mcS_i')$.
\end{enumerate}
\end{Proposition}

\begin{proof}
For $p_i\neq\infty$, this follows directly from Lemma~\ref{MainLemma}.
For $p_i=\infty$, proceed as follows: 
If $a=0$, then $a\in\varphi_i(F)$ and $a\in R(\mcS_i)$.
If $a\in\varphi_i(F)\setminus\{0\}$, then 
Lemma~\ref{MainLemmaInfty}(1) implies that for some $u\in F^\times$, $a-u^{-2}\geq_\mfP -u^{-2}$ for all $\mfP\in\mcS_i$,
so $a\in R(\mcS_i)$.
If $F$ is $\PmcSC$ and $a\in R(\mcS_i)\setminus\{0\}$,
then a simple calculation shows that with $u=a^{-1}(a+1)\in F$, also $a-u^{-2}\in R(\mcS_i)$.
Hence, by Lemma~\ref{MainLemmaInfty}(2),
$\varphi_i(a)$ is satisfied in $F$.
\end{proof}

\section{Holomorphy Domains in $\PSCL$ Fields}
\label{sec3}

\noindent
Now we apply the general construction of the previous section to
the fields we are interested in.
We continue to work in Setting~\ref{SettingPSCC} 
and, for the rest of this paper, make the following additional assumptions:

\begin{enumerate}
 \item[\textbullet] {\bf For $\mfp\in S$ with $p_\mfp\neq\infty$, fix $\pi_\mfp\in K$ with $v_\mfp(\pi_\mfp)=1$.}
 \item[\textbullet] {\bf For $\mfpinS$ with $p_\mfp=\infty$, let $\pi_\mfp=-1$.}
\end{enumerate}

\begin{Lemma}\label{TimpliesS}
Let $\tau'\leq\tau$ be a relative type and write
$S=\{\mfp_1,\dots,\mfp_n\}$, $\tau=(e,f)$, and $\tau'=(e',f')$.
Let
$\mcS_i=\mcS_{\mfp_i}^\tau(F)$,
$\mcS=\biguplus_{i=1}^n\mcS_i$,
$p_i=p_{\mfp_i}$, $e_i=e$, $f_i=ff_{\mfp_i}$, $p_i'=p_{\mfp_i}$, $e_i'=e'$, $f_i'=f'f_{\mfp_i}$.
Then there exist $\pi_1,\dots,\pi_n\in K$
such that the conditions of Setting~\ref{settingS} are satisfied.
\end{Lemma}

\begin{proof}
The existence of $\pi_i$ follows from the weak approximation theorem
applied to the finite set $S$, see e.g.~\cite[1.1.3]{EnglerPrestel}.
\end{proof}

\begin{Proposition}\label{MainTheoremT}
Let $\mfp\in S$ and $\tau'\leq\tau$ a relative type.
There exists an existential $\mcLr(K)$-formula $\theta_{R,\mfp}^{\tau'}(z)$ 
that satisfies the following:
\begin{enumerate}
 \item $\theta_{R,\mfp}^{\tau'}(F)\subseteq R_\mfp^{\tau'}(F)$.
 \item If $F$ is $\PStauCL$, then $\theta_{R,\mfp}^{\tau'}(F)=R_\mfp^{\tau'}(F)$.
\end{enumerate}
\end{Proposition}

\begin{proof}
Apply Lemma \ref{TimpliesS} and assume $\mfp=\mfp_i$.
Then the corresponding formula $\varphi_i$ of
Proposition~\ref{MainTheoremS}
satisfies the claim.
\end{proof}

This also proves Theorem \ref{thm2} of the introduction.
Indeed, if $p_i\neq\infty$, then the formula $\theta_{R,\mfp}^{\tau}(z)$ 
is already diophantine (and independent of $F$).
In the case $p_i=\infty$, the formula $\theta_{R,\mfp}^{\tau}(z)$ 
is of the form 
$$
 (\exists u\neq 0)(\exists x,y)(f(z,u,x,y)=0)
$$ 
with $f\in K[Z,U,X,Y]$ independent of $F$,
which for a $\PStauCC$ field is equivalent to the diophantine formula
$$
 (\exists u,v,x,y)(f(z,u,x,y)^2+(uv-1)^2=0)
$$ 
(note that if $F$ is not real, then already $\theta_{R,\mfp}^{\tau}(F)=F$).

\begin{Definition}\label{DefKochen}
Let $\mfp\in S$ and $\tau'=(e',f')\leq\tau$.
If $p_\mfp\neq\infty$,
let $q=p^{f'f_\mfp}$,
and define the {\bf $\mfp$-adic Kochen operator} over $K$ of type ${\tau'}$ by
$$
 \gamma_\mfp^{\tau'}(x)=\frac{1}{\pi_\mfp}\cdot((x^q-x)-(x^q-x)^{-1})^{-e'}
$$
if this expression is well defined, and $\gamma_\mfp^{\tau'}(x)=0$ otherwise.
Define the {\bf $\mfp$-adic Kochen ring} over $K$ of type ${\tau'}$ of $F$ by
$$
 \Gamma_\mfp^{\tau'}(F)=\left\{ \frac{b}{1+\pi_\mfp c} \colon b,c\in\mcO_\mfp[\gamma_\mfp^{\tau'}(F)],\; 1+\pi_\mfp c\neq 0 \right\}.
$$
If $p_\mfp=\infty$, let $\gamma_\mfp^{\tau'}(x)=\gamma(x)=x^2$
and $\Gamma_\mfp^{\tau'}(F)=\mathcal{O}_\mfp[\gamma(F)]$,
the semiring generated by $\gamma(F)$ over $\mathcal{O}_\mfp$.
\end{Definition}

\begin{Lemma}\label{Kochenring}
Let $\mfp\in S$ and $\tau'\leq\tau$.
Then $\mcS_\mfp^{\tau'}(F)\neq\emptyset$ if and only if $\pi_\mfp^{-1}\notin\Gamma_\mfp^{\tau'}(F)$.
In that case, if $p_\mfp\neq\infty$, then 
$R_\mfp^{\tau'}(F)$ is the integral closure of $\Gamma_\mfp^{\tau'}(F)$;
if $p_\mfp=\infty$, then $R_\mfp^{\tau'}(F)=\Gamma_\mfp^{\tau'}(F)$.
\end{Lemma}

\begin{proof}
For the case $p_\mfp\neq\infty$ see \cite[6.4, 6.8, 6.9]{PR}.
For the case $p_\mfp=\infty$
note that
if $-1\notin\Gamma_\mfp^{\tau'}(F)$, then 
$\Gamma_\mfp^{\tau'}(F)$ is a pre-positive cone,
so $\Gamma_\mfp^{\tau'}(F)=R_\mfp^{\tau'}(F)$,
see \cite[1.6]{Prestel1}.
\end{proof}

\begin{Definition}\label{DefTholomp}
For $\mfp\in S$ and $\tau'\leq\tau$, let $T_{R,\mfp}^{\tau'}$ be the $\mcLr(K)$-theory consisting of the following sentences.
\begin{enumerate}
 \item A recursive set of sentences stating that $\theta_{R,\mfp}^{\tau'}$ defines 
an integrally closed ring (if $p_\mfp\neq\infty$)
         resp.~a semiring (if $p_\mfp=\infty$).
 \item For every $a\in\mathcal{O}_\mfp$ the sentence
$$
 \theta_{R,\mfp}^{\tau'}(a).
$$
 \item The sentence
$$
 (\forall x)(\theta_{R,\mfp}^{\tau'}(\gamma_\mfp^{\tau'}(x))).
$$
 \item If $p_\mfp\neq\infty$, the sentence
$$
 (\forall x)(\theta_{R,\mfp}^{\tau'}(x)\wedge 1+\pi_\mfp x\neq0\rightarrow\theta_{R,\mfp}^{\tau'}((1+\pi_\mfp x)^{-1})).
$$
 \item The sentence
$$
 \theta_{R,\mfp}^{\tau'}(\pi_\mfp^{-1})\rightarrow(\forall x)(\theta_{R,\mfp}^{\tau'}(x)).
$$
\end{enumerate}
\end{Definition}

\begin{Proposition}\label{Tholom}
Let $\mfp\in S$ and $\tau'\leq\tau$.
Then $F$ satisfies $T_{R,\mfp}^{\tau'}$ if and only if
the formula $\theta_{R,\mfp}^{\tau'}$ defines the holomorphy domain $R_\mfp^{\tau'}(F)$ in $F$.
\end{Proposition}

\begin{proof}
This follows from Proposition~\ref{MainTheoremT}(1), Lemma \ref{Kochenring} and the definition of $\Gamma_\mfp^{\tau'}(F)$.
\end{proof}

\section{Quantification over Classical Primes}
\label{sec4}

\noindent
In this section we translate first-order 
statements concerning the classical primes of $F$ to statements
about $F$ and the corresponding holomorphy domains.

\begin{Lemma}\label{Hp}
Let $\mfp\in S$ with $p_\mfp\neq\infty$, and $\tau'\leq\tau=(e,f)$.
For $a\in F$ let 
$$
 H_\mfp^{\tau'}(a)=\{\mfP\in\mcS_\mfp^{\tau'}(F) \colon a\in\mcO_\mfP\}.
$$
Then the following holds:
\begin{enumerate}
\item If $a,b\in F$, then $H_\mfp^{\tau'}(a)\cap H_\mfp^{\tau'}(b)=H_\mfp^{\tau'}(a^{2^{e}}+\pi_\mfp b^{2^{e}})$.
\item If $a\in F^\times$, then $\mcS_\mfp^{\tau'}(F)\setminus H_\mfp^{\tau'}(a) = H_\mfp^{\tau'}((\pi_\mfp a^{2^{e}})^{-1})$.
\item If $P(Z_1,\dots,Z_n)$ is a boolean polynomial\footnote{i.e.~a term in the language of boolean algebras, cf.~\cite[Chapter 7.6]{FJ3}},
        then there exists a rational function
        $r(\mathbf{X})\in\mathbb{Q}(\pi_\mfp)(X_1,\dots,X_n)$ 
        independent of $F$ such that for all $a_1,\dots,a_n\in F$,
        \begin{eqnarray}\label{eqnboolean}
          P(H_\mfp^{\tau'}(a_1),\dots,H_\mfp^{\tau'}(a_n))=H_\mfp^{\tau'}(r(a_1,\dots,a_n)).
        \end{eqnarray}
\end{enumerate}
\end{Lemma}

\begin{proof}
(1):
Let $\mfP\in\mcS_\mfp^{\tau'}(F)$.
If $v_\mfP(a)\geq0$ and $v_\mfP(b)\geq 0$, then
$v_\mfP(a^{2^e}+\pi_\mfp b^{2^e})\geq0$.
If $v_\mfP(a)<0$ or $v_\mfP(b)<0$, then
$v_\mfP(a^{2^{e}}+\pi_\mfp b^{2^{e}})=\min\{v_\mfP(a^{2^e}),v_\mfP(\pi_\mfp b^{2^e})\}<0$,
since $0<v_\mfP(\pi_\mfp)\leq e'\leq e < 2^e$.

(2):
Let $\mfP\in\mcS_\mfp^{\tau'}(F)$.
If $v_\mfP(a)\geq 0$, then $v_\mfP(\pi_\mfp a^{2^e})\geq v_\mfP(\pi_\mfp)>0$, so
$v_\mfP((\pi_\mfp a^{2^e})^{-1})< 0$.
If $v_\mfP(a)<0$, then 
$v_\mfP(a^{2^e})\leq-2^e<-e\leq-v_\mfp(\pi_\mfp)$,
so $v_\mfP((\pi_\mfp a^{2^e})^{-1})\geq 0$.

(3):
If $\mcS_\mfp^{\tau'}(F)=\emptyset$, then every $r(\mathbf{X})$ satisfies (\ref{eqnboolean}).
Thus, assume that $\mcS_\mfp^{\tau'}(F)\neq\emptyset$ and hence $a^{2^e}+\pi_\mfp\neq0$
for every $a\in F$.
By (1),
$H_\mfp^{\tau'}(a)=H_\mfp^{\tau'}(a)\cap H_\mfp^{\tau'}(1)=H_\mfp^{\tau'}(a^{2^e}+\pi_\mfp)$.
Hence,
the set of boolean polynomials $P(\mathbf{Z})$
for which 
there exists a rational function
$r(\mathbf{X})\in\mathbb{Q}(\pi_\mfp)(X_1,\dots,X_n)$ such that
$r(\mathbf{a})\notin\{0,\infty\}$ and
(\ref{eqnboolean})
hold for all $a_1,\dots,a_n\in F$
contains $Z_1,\dots,Z_n$.
By (1), it is closed under intersections.
By (2), it is closed under complements.
Hence, it contains all boolean polynomials.
\end{proof}

\begin{Remark}
In what comes,
the predicate symbol $R$ of the language $\mcL_R$
will be used in two different ways.
It will interpret either a valuation ring resp.~positive cone $\mcO_\mfP$,
or a holomorphy domain $R_\mfp^{\tau'}(F)$. 
We write $(F,\mcO_\mfP)$ and $(F,R_\mfp^{\tau'}(F))$, respectively, for the corresponding structures.

Note that formally
we work in the language $\mcLr$ of rings,
i.e.~there is no function $\cdot^{-1}$ in our language.
However, it is common to use this function in first-order formulas when working in fields,
knowing that it can always be eliminated by introducing
either an existential or a universal quantifier.
\end{Remark}

The following proposition makes explicit some ideas from \cite[p.~154]{PrestelPRC}
and \cite[proof of Theorem 4.01]{Grob}.

\begin{Proposition}\label{quantifyprime}
Let $\mfp\in S$ and $\tau'\leq\tau$.

\begin{enumerate}
\item[$(1)$]
There exists a recursive map $\varphi(\mathbf{x})\mapsto\varphi_{\mfp,\exists}^{\tau'}(\mathbf{x})$ 
from existential $\mathcal{L}_R$-formulas to
$\mathcal{L}_R(\pi_\mfp)$-formulas such that
for every extension $F/K$ and elements $a_1,\dots,a_n\in F$
the following statements are equivalent:
\begin{enumerate}
 \item[{\rm(1a)}] There exists $\mfP\in\mcS_\mfp^{\tau'}(F)$ such that $(F,\mathcal{O}_\mfP)\models\varphi(\mathbf{a})$.
 \item[{\rm(1b)}] $(F,R_\mfp^{\tau'}(F))\models\varphi_{\mfp,\exists}^{\tau'}(\mathbf{a})$.
\end{enumerate}
\item[$(2)$]
There exists a recursive map $\varphi(\mathbf{x})\mapsto\varphi_{\mfp,\forall}^{\tau'}(\mathbf{x})$ 
from universal $\mathcal{L}_R$-formulas to
$\mathcal{L}_R(\pi_\mfp)$-formulas such that
for every extension $F/K$ and elements $a_1,\dots,a_n\in F$
the following statements are equivalent:
\begin{enumerate}
 \item[{\rm(2a)}] $(F,\mathcal{O}_\mfP)\models\varphi(\mathbf{a})$ for all $\mfP\in\mcS_\mfp^{\tau'}(F)$.
 \item[{\rm(2b)}] $(F,R_\mfp^{\tau'}(F))\models\varphi_{\mfp,\forall}^{\tau'}(\mathbf{a})$.
\end{enumerate}
\end{enumerate}
\end{Proposition}

\begin{proof}
First of all, note that we can get (2) from (1) 
via $\varphi_{\mfp,\forall}^{\tau'}:\Leftrightarrow \neg(\neg\varphi)_{\mfp,\exists}^{\tau'}$.
Thus, it suffices to prove (1).

\CASE{Part A1: Case $p_\mfp\neq\infty$}
First assume that $\varphi(\mathbf{x})$ is of the simple form
$$
 \bigwedge_i(f_i(\mathbf{x})\in R)\wedge\bigwedge_i(g_i(\mathbf{x})\notin R),
$$
where $f_i,g_i\in\mathbb{Z}[\mathbf{X}]$ for all $i$.
Let 
$$
 \mathcal{H}(\mathbf{a})=\bigcap_i H_\mfp^{\tau'}(f_{i}(\mathbf{a}))\cap\bigcap_i(\mcS_\mfp^{\tau'}(F)\setminus H_\mfp^{\tau'}(g_{i}(\mathbf{a}))).
$$
By Lemma~\ref{Hp}(3) there
exists a rational function $r\in\mathbb{Q}(\pi_\mfp)(\mathbf{X})$
independent of $F$ and $\mathbf{a}$
such that 
$H_\mfp^{\tau'}(r(\mathbf{a}))=\mcS_\mfp^{\tau'}(F)\setminus\mathcal{H}(\mathbf{a})$.
Then (1a) holds if and only if $\mathcal{H}(\mathbf{a})\neq\emptyset$,
that is, $H_\mfp^{\tau'}(r(\mathbf{a}))\neq\mcS_\mfp^{\tau'}(F)$,
which is equivalent to $r(\mathbf{a})\notin R_\mfp^{\tau'}(F)$.
Thus, if we let $\varphi_{\mfp,\exists}^{\tau'}(\mathbf{x})$ be the formula
$\neg(r(\mathbf{x})\in R)$, then the claim follows.

\CASE{Part A2: Conclusion of the proof for $p_\mfp\neq\infty$}
Now assume that $\varphi(\mathbf{x})$ is an arbitrary existential $\mathcal{L}_R$-formula in prenex disjunctive normal form,
i.e.~$\varphi(\mathbf{x})$ is of the form
\begin{eqnarray*}
 (\exists y_1,\dots,y_m)\bigvee_j&& [\bigwedge_i(f_{ij}(\mathbf{x},\mathbf{y})\in R)\wedge\bigwedge_i(g_{ij}(\mathbf{x},\mathbf{y})\notin R)\wedge\\
   &&\wedge\bigwedge_i(h_{ij}(\mathbf{x},\mathbf{y})=0)\wedge\bigwedge_i(k_{ij}(\mathbf{x},\mathbf{y})\neq0)],
\end{eqnarray*}
where $f_{ij},g_{ij},h_{ij},k_{ij}\in\mathbb{Z}[\mathbf{X},\mathbf{Y}]$.
Let $\varphi_j(\mathbf{x},\mathbf{y})$ be the formula
$$
 \bigwedge_i(f_{ij}(\mathbf{x},\mathbf{y})\in R)\wedge\bigwedge_i(g_{ij}(\mathbf{x},\mathbf{y})\notin R).
$$
Then $\varphi_j$ is of the special form considered in {\sc Part A1}.
Let $\varphi_{\mfp,\exists}^{\tau'}(\mathbf{x})$ be the formula
$$
 (\exists y_1,\dots,y_m) \bigvee_j [(\varphi_j)_{\mfp,\exists}^{\tau'}(\mathbf{x},\mathbf{y})\wedge\bigwedge_i(h_{ij}(\mathbf{x},\mathbf{y})=0)\wedge\bigwedge_i(k_{ij}(\mathbf{x},\mathbf{y})\neq0)].
$$
Then $\varphi_{\mfp,\exists}^{\tau'}$ satisfies the claim.

\CASE{Part B1: Case $p_\mfp=\infty$}
First assume that $\varphi(\mathbf{x})$ is of the form
$$
 \bigwedge_i(f_i(\mathbf{x})\in R)
$$
where $f_1,\dots,f_m\in\mathbb{Z}[\mathbf{X}]$.
Assume that {\rm(1a)} holds.
Then there exists an ordering $\mfP\in\mcS_\mfp^{\tau'}(F)$ with 
$f_1(\mathbf{a})\geq_\mfP 0,\dots,f_m(\mathbf{a})\geq_\mfP 0$.
Hence, 
$R_\mfp^{\tau'}(F)[f_1(\mathbf{a}),\dots,f_m(\mathbf{a})]$, 
the semiring generated by $f_1(\mathbf{a}),\dots,f_m(\mathbf{a})$ over $R_\mfp^{\tau'}(F)$,
is contained in $\mcO_\mfP$.
In particular,
\begin{eqnarray}
 \label{eqn2} R_\mfp^{\tau'}(F)[f_1(\mathbf{a}),\dots,f_m(\mathbf{a})]\cap(-R_\mfp^{\tau'}(F))&=&\{0\},
\end{eqnarray}
so if $\varphi_{\mfp,\exists}^{\tau'}(\mathbf{x})$ is the formula
\begin{eqnarray*}
 (\forall s_1,\dots,s_r\in R)&&(-\sum_{j=1}^r s_jf_1(\mathbf{x})^{k_{j,1}}\cdots f_m(\mathbf{x})^{k_{j,m}}\in R \\
							&& \quad\rightarrow \sum_{j=1}^r s_jf_1(\mathbf{x})^{k_{j,1}}\cdots f_m(\mathbf{x})^{k_{j,m}}= 0 ),
\end{eqnarray*}
where $r=2^m$, and $(k_{j,1},\dots,k_{j,m})$ ranges over $\{0,1\}^m$,
then {\rm(1a)} implies {\rm(1b)}.

Conversely,
suppose that {\rm(1b)} holds.
Then,
since $F^2\subseteq R_\mfp^{\tau'}(F)$,
$\varphi_{\mfp,\exists}^{\tau'}(\mathbf{a})$
holds in $F$ even if
$(k_{j,1},\dots,k_{j,m})$ ranges over any finite subset of $(\Zge)^m$.
Hence, (\ref{eqn2}) holds.
Thus, $R_\mfp^{\tau'}(F)[f_1(\mathbf{a}),\dots,f_m(\mathbf{a})]$ is a pre-positive cone,
and hence there exists
an ordering $\mfP\in\mcS_\mfp^{\tau'}(F)$ with 
$f_1(\mathbf{a})\geq_\mfP 0,\dots,f_m(\mathbf{a})\geq_\mfP 0$, \cite[1.6]{Prestel1}.
That is, (1a) holds.

\CASE{Part B2: Conclusion of the proof for $p_\mfp=\infty$}
Now assume that $\varphi(\mathbf{x})$ is an arbitrary existential $\mathcal{L}_R$-formula in prenex disjunctive normal form.
Replace $x\notin R$ by $(-x\in R)\wedge(x\neq 0)$ 
and conclude the proof as in {\sc Part A2}.
\end{proof}

\section{Quantification over Classical Closures}
\label{sec:quantifyCC}

\noindent
We use the quantification over classical primes of the previous section 
to quantify over classical closures.

We want to make use of the following purely model theoretic lemma, which we prove due to lack of a reference.
Let $T_0\subseteq T$ be theories in a language $\mcL$.
Write $(M_0,M)\models(T_0,T)$ to indicate that $M$ is a model of $T$, and $M_0$ is a substructure of $M$ and a model of $T_0$.
Let $\Delta(M_0)$ denote the quantifier-free diagram of $M_0$ in the language $\mathcal{L}(M_0)$.

\begin{Lemma}\label{Itay}
If $(M_0,M)\models(T_0,T)$ implies that $T\cup\Delta(M_0)$ is complete, then
there exists a map $\varphi(\mathbf{x})\mapsto\varphi^0(\mathbf{x})$ from $\mcL$-formulas to universal $\mcL$-formulas such that
for every $\mcL$-formula $\varphi(\mathbf{x})$,
the following holds:
\begin{enumerate}
\item $T\models\forall\mathbf{x} (\varphi(\mathbf{x})\leftrightarrow\varphi^0(\mathbf{x}))$.
\item If $(M_0,M)\models(T_0,T)$ and $\mathbf{a}\in M_0^r$, then $M\models\varphi^0(\mathbf{a})$ if and only if $M_0\models\varphi^0(\mathbf{a})$.
\end{enumerate}
If both $T$ and $T_0$ are recursively enumerable, then the map $\varphi(\mathbf{x})\mapsto\varphi^0(\mathbf{x})$ is recursive.
\end{Lemma}

\begin{proof}[Proof (Itay Kaplan).]
Let
\begin{eqnarray*}
\Gamma = \{ \alpha(\mathbf{x}) &:& \alpha\mbox{ universal $\mathcal{L}$-formula, and if } (M_0,M)\models(T_0,T) \\ 
 &&  \mbox{ and } \mathbf{a}\in M_0^r, \mbox{ then } M\models\alpha(\mathbf{a})\mbox{ if and only if }M_0\models\alpha(\mathbf{a})\}
\end{eqnarray*}
and
\begin{eqnarray*}
 \Sigma = \{ \alpha(\mathbf{x}) &:& \alpha\mbox{ universal $\mathcal{L}$-formula, }T\models\varphi\rightarrow\alpha \}.
\end{eqnarray*}
Since $T_0\subseteq T$, the assumption implies that $T$ is model complete.
Thus, we can assume without loss of generality that $\varphi$ is universal, \cite[3.4.12(d)]{Marker}.
If $A=\{\alpha_1,\dots,\alpha_n\}\subseteq\Sigma$ is a finite set and 
$\beta_A:=\varphi\wedge\alpha_1\wedge\dots\wedge\alpha_n$, then 
$T\models\varphi\leftrightarrow\beta_A$.
Hence, if $\beta_A\in\Gamma$, then $\varphi^0:=\beta_A$ satisfies (1) and (2).
Suppose that this does not happen, that is, 
$\beta_A\notin\Gamma$ for every finite set $A\subseteq\Sigma$.
Since $\beta_A$ is universal, this means that there exists $(M_0,M)\models(T_0,T)$ and $\mathbf{a}\in M_0^r$
with $M_0\models\beta_A(\mathbf{a})$ and $M\models\neg\beta_A(\mathbf{a})$.
Since $T\models\varphi\leftrightarrow\beta_A$, we get that $M\models\neg\varphi(\mathbf{a})$.

Let $\mathcal{L}^0=\mathcal{L}\cup\{P,\mathbf{c}\}$, where $P$ is a unary predicate symbol and $\mathbf{c}=(c_1,\dots,c_r)$ are constant symbols.
Let the $\mathcal{L}^0$-theory $T^0$ consist of the theory $T$ and
the statement that $P$ defines a substructure that contains $\mathbf{c}$ and is a model of $T_0$.
Let $T^1$ consists of $T^0$, the sentence $\neg\varphi(\mathbf{c})$,
and for every $\alpha\in\Sigma$ the statement that $\alpha(\mathbf{c})$ holds in the substructure defined by $P$.

By the above assumption, every finite subset of $T^1$ is consistent.
Therefore the compactness theorem implies that $T^1$ has a model. That is, there exists $(M_0,M)\models(T_0,T)$
and $\mathbf{a}\in M_0^r$ such that $M\models\neg\varphi(\mathbf{a})$ and $M_0\models\alpha(\mathbf{a})$ for every $\alpha\in\Sigma$.

Since by assumption $T\cup\Delta(M_0)$ is a complete $\mathcal{L}(M_0)$-theory, 
$T\cup\Delta(M_0)\models\neg\varphi(\mathbf{a})$.
Thus there exists $\psi(\mathbf{x},\mathbf{y})$ and $\mathbf{b}\in M_0^s$ such that $\psi(\mathbf{a},\mathbf{b})\in\Delta(M_0)$
and $T\models\forall \mathbf{x}\forall \mathbf{y}(\psi(\mathbf{x},\mathbf{y})\rightarrow\neg\varphi(\mathbf{x}))$.
Therefore, $(\forall\mathbf{y})(\neg\psi(\mathbf{x},\mathbf{y}))\in\Sigma$.
By construction of $M_0$, this implies that $M_0\models(\forall\mathbf{y})(\neg\psi(\mathbf{a},\mathbf{y}))$,
contradicting $\psi(\mathbf{a},\mathbf{b})\in\Delta(M_0)$.
This contradiction shows that $\beta_A\in\Gamma$ for some $A$, as desired.

If both $T_0$ and $T$ are recursively enumerable, then so is $T^0$.
Since a universal $\mathcal{L}$-formula $\beta$ is in $\Gamma$ if and only if $T^0\vdash\beta(\mathbf{c})\leftrightarrow\beta^P(\mathbf{c})$,
where $\beta^P$ is $\beta$ with all quantifiers restricted to $P$,
$\Gamma$ is recursively enumerable. 
Thus one can recursively determine a universal $\mathcal{L}$-formula $\beta\in\Gamma$ with $T\vdash\varphi\leftrightarrow\beta$. 
Therefore, the map $\varphi\mapsto\varphi^0$ can be chosen recursive.
\end{proof}

Note that the assumption of Lemma \ref{Itay} is satisfied in particular if
$T$ is the {\bf model completion} of $T_0$, cf.~\cite[3.4.14]{Marker}.
Also, the cases $T_0=\emptyset$ and $T_0=T$ of Lemma \ref{Itay} are well-known characterizations
of quantifier elimination resp.~model completeness.

\begin{Proposition}\label{almostQE}
For every type $\tau_1=(p,e_1,f_1)$ there exists a recursive map $\varphi(\mathbf{x})\mapsto\bar{\varphi}^{\tau_1}(\mathbf{x})$
from $\mcL_R$-formulas to universal $\mcL_R$-formulas 
with the following properties:
\begin{enumerate}
\item For every classically closed field $(F',\mfP)$ with $\tp(\mfP)=\tau_1$,
$$
 (F',\mcO_{\mfP}) \models (\forall \mathbf{x})(\varphi(\mathbf{x})\leftrightarrow\bar{\varphi}^{\tau_1}(\mathbf{x})).
$$
\item If $\mfP$ is a quasi-local prime of a field $F$ with $\tp(\mfP)=\tau_1$ and $\mathbf{a}\in F^r$, then
$(F,\mcO_\mfP)\models\bar{\varphi}^{\tau_1}(\mathbf{a})$ if and only if $(F_\mfP,\mcO_{F_\mfP})\models \bar{\varphi}^{\tau_1}(\mathbf{a})$.
\item If $\mfP$ is a prime of a field $F$ with $\tp(\mfP)\leq\tau_1$ but $\tp(\mfP)\neq\tau_1$,
 and $\mathbf{a}\in F^r$, then
$(F,\mcO_\mfP)\models\bar{\varphi}^{\tau_1}(\mathbf{a})$.
\end{enumerate}
\end{Proposition}

\begin{proof}
For $p=\infty$, this follows directly from quantifier elimination
for real closed fields.

For $p\neq\infty$, apply Lemma \ref{Itay} with $T$ the theory of $p$-adically closed fields of type $\tau_1$
and $T_0$ the theory of $p$-valued fields of type $\tau_1$ with value group a $\mathbb{Z}$-group.
The assumptions of Lemma \ref{Itay} are satisfied by \cite[3.2, 3.4, 5.1]{PR}
(in fact, this shows that $T$ is the model completion of $T_0$).
Therefore, if we let $\bar{\varphi}^{\tau_1}(\mathbf{x})$ be
the formula $\varphi^0(\mathbf{x})$ of Lemma \ref{Itay}, then (1) and (2) are satisfied.
In order to satisfy also (3), let $\psi$ be the existential $\mathcal{L}_R$-sentence 
$$
 (\exists x\in R)(x^{-1}\notin R\wedge px^{-e_1}\in R)\wedge(\exists x\in R)(\Phi_{p^{f_1}-1}(x)^{-1}\notin R)
$$
where $\Phi_{p^{f_1}-1}$ is the $(p^{f_1}-1)$-th cyclotomic polynomial.
Note that $(F,\mcO_\mfP)\models\psi$ if and only if $\tp(\mfP)\geq\tau_1$.
Thus, if we let $\bar{\varphi}^{\tau_1}(\mathbf{x})$ be the universal $\mathcal{L}_R$-formula $\psi\rightarrow\varphi^0(\mathbf{x})$,
then (1)-(3) are satisfied.

Since the theories in question are axiomatized by recursive sets of sentences, and hence are recursively enumerable,
the map $\varphi(\mathbf{x})\mapsto\varphi^0(\mathbf{x})$ is recursive by Lemma \ref{Itay}. Therefore, also the map $\varphi(\mathbf{x})\mapsto\bar{\varphi}^{\tau_1}(\mathbf{x})$ is recursive.
\end{proof}

\begin{Lemma}\label{quantifylocalLemma}
Let $\mfp\in S$ and $\tau'\leq\tau$.
There exists a recursive map 
$\varphi(\mathbf{x})\mapsto\hat{\varphi}_{\mfp,\forall,R}^{\tau'}(\mathbf{x})$
from $\mcLr$-formulas to $\mathcal{L}_R(K)$-formulas
such that
for every extension $F/K$ and elements $a_1,\dots,a_m\in F$
the following holds:
\begin{enumerate}
 \item If $F^\prime\models\varphi(\mathbf{a})$ holds for all $F^\prime\in\CC_\mfp^{\tau'}(F)$ with $\tp(\mfP_{F'}/\mfp)=\tau'$, 
          then $(F,R_\mfp^{\tau'}(F))\models\hat{\varphi}_{\mfp,\forall,R}^{\tau'}(\mathbf{a})$.
 \item If $(F,R_\mfp^{\tau'}(F))\models\hat{\varphi}_{\mfp,\forall,R}^{\tau'}(\mathbf{a})$
         and $\mfP\in\mcS_\mfp^{\tau'}(F)$ with $\tp(\mfP/\mfp)=\tau'$ is quasi-local,
         then $F_\mfP\models\varphi(\mathbf{a})$.
\end{enumerate}
\end{Lemma}

\begin{proof}
Write $\tau'=(e',f')$ and let $\tau_1=(p_\mfp,e'e_\mfp,f'f_\mfp)$.
Let $\psi(\mathbf{x})$ be the formula $\bar{\varphi}^{\tau_1}(\mathbf{x})$ of Proposition \ref{almostQE}
and let $\hat{\varphi}_{\mfp,\forall,R}^{\tau'}(\mathbf{x})$
be the formula $\psi_{\mfp,\forall}^{\tau'}(\mathbf{x})$ that 
Proposition \ref{quantifyprime}
attaches to $\psi(\mathbf{x})$.
Then $\hat{\varphi}_{\mfp,\forall,R}^{\tau'}(\mathbf{x})$
satisfies the claim.
\end{proof}

\begin{Proposition}\label{quantifylocal}
Let $\mfp\in S$.
There exists a recursive map
$\varphi(\mathbf{x})\mapsto\hat{\varphi}_{\mfp,\forall}^\tau(\mathbf{x})$
from $\mcLr$-formulas to $\mcLr(K)$-formulas
such that
for every extension $F/K$
that satisfies $T_{R,\mfp}^{\tau'}$ for all $\tau'\leq\tau$,
and for all elements $a_1,\dots,a_m\in F$
the following holds:
\begin{enumerate}
 \item If $F^\prime\models\varphi(\mathbf{a})$ for all $F^\prime\in\CC_\mfp^\tau(F)$,
        then $F\models\hat{\varphi}_{\mfp,\forall}^\tau(\mathbf{a})$.
 \item If $F\models\hat{\varphi}_{\mfp,\forall}^\tau(\mathbf{a})$ and $\mfP\in\mcS_\mfp^\tau(F)$ is quasi-local,
         then $F_\mfP\models\varphi(\mathbf{a})$.
\end{enumerate}
\end{Proposition}

\begin{proof}
For $\tau'\leq\tau$ let $\psi^{\tau'}(\mathbf{x})$ be the formula $\hat{\varphi}_{\mfp,\forall,R}^{\tau'}(\mathbf{x})$
of Lemma~\ref{quantifylocalLemma}
with all occurences of $x\in R$ replaced by the formula
$\theta_{R,\mfp}^{\tau'}(x)$
of Proposition \ref{MainTheoremT}.
Let $\hat{\varphi}_{\mfp,\forall}^\tau(\mathbf{x})$
be the formula
$\bigwedge_{\tau'\leq\tau}\psi^{\tau'}$.
Then $\hat{\varphi}_{\mfp,\forall}^\tau(\mathbf{x})$ satisfies the claim.
This follows from
Lemma~\ref{quantifylocalLemma}
and Proposition \ref{Tholom}.
\end{proof}

\section{Axiomatization of $\PStauCC$ Fields}
\label{sec5}

\noindent
We use the results of the previous section
to axiomatize
the $\PStauCC$ property.

\begin{Definition}\label{DefTPSCC}
Construct an $\mcLr(K)$-theory  $T_{\PStauCC}$ as follows:
Let 
$$
 f_n(\mathbf{T},\mathbf{Z})=\sum_{\bfalpha} T_{\bfalpha} Z_1^{\alpha_1}\cdots Z_n^{\alpha_n} \in\mathbb{Z}[\mathbf{T},\mathbf{Z}]
$$
be the general polynomial in $n$ variables $Z_1,\dots,Z_n$ of degree $n$ with coefficients $\mathbf{T}$.
Here $\bfalpha$ runs over all $n$-tuples $\bfalpha=(\alpha_1,\dots,\alpha_n)$, $\alpha_i\in\mathbb{Z}_{\geq 0}$,
$\sum_{i=1}^n \alpha_i \leq n$.

For $n\in\mathbb{N}$, let $\psi_n(\mathbf{x},\mathbf{y})$ be an $\mcLr$-formula stating that the 
polynomial $f_n(\mathbf{x},\mathbf{Z})$ with coefficients $\mathbf{x}$
is absolutely irreducible (see for example \cite[Chapter 11.3]{FJ3}),
and all singular points on the affine hypersurface defined by this polynomial
lie on the subvariety defined by the polynomial $f_n(\mathbf{y},\mathbf{Z})$ with coefficients $\mathbf{y}$.
Let $\eta_n(\mathbf{x},\mathbf{y})$ be the $\mcLr$-formula
$$
 (\exists \mathbf{z})( f_n(\mathbf{x},\mathbf{z})=0 \wedge f_n(\mathbf{y},\mathbf{z})\neq 0)
$$
stating that the polynomial with coefficients $\mathbf{x}$
has a zero which is not a zero of the
polynomial with coefficients $\mathbf{y}$.
Let 
$(\hat{\eta}_n)_{\mfp,\forall}^\tau(\mathbf{x},\mathbf{y})$
be the $\mcLr(K)$-formula
that Proposition~\ref{quantifylocal}
attaches to $\eta_n$, and
let $\varphi_n$ be the $\mcLr(K)$-sentence 
$$
 (\forall \mathbf{x},\mathbf{y})[(\psi_n(\mathbf{x},\mathbf{y})\wedge\bigwedge_{\mfp\in S} (\hat{\eta}_n)_{\mfp,\forall}^\tau(\mathbf{x},\mathbf{y})) \rightarrow \eta_n(\mathbf{x},\mathbf{y})].
$$
Let $T_{\PStauCC}$ consist of the following sentences:
\begin{enumerate}
 \item For every $\mfp\in S$ and $\tau'\leq\tau$, the theory $T_{R,\mfp}^{\tau'}$.
 \item For every $n\in\mathbb{N}$, the sentence $\varphi_n$.
\end{enumerate}
\end{Definition}

\begin{Lemma}\label{varieties}
Let $\mfpinS$ and $F^\prime\in\CC_\mfp^\tau(F)$,
and let $V$ be a smooth $F$-variety.
Then $V(F^\prime)\neq\emptyset$ if and only if
there exists $\mfP\in\mcS_\mfp^\tau(F'(V))$ with $\tp(\mfP)=\tp(F')$.
\end{Lemma}

\begin{proof}
For $p_\mfp\neq\infty$, this follows from
\cite[7.8]{PR}; for $p_\mfp=\infty$, it follows from \cite[3.13]{Prestel1}.
\end{proof}

\begin{Proposition}\label{PsefCCAx}
The field $F$ satisfies $T_\PStauCC$ if and only if $F$ is $\PStauCC$.
\end{Proposition}

\begin{proof}
First assume that $F$ is $\PStauCC$.
Then $F$ is also $\PStauCL$ (cf.~Remark \ref{PSCCimpliesPSC})
and hence satisfies (1) by
Proposition~\ref{MainTheoremT}
and Proposition~\ref{Tholom}.
For all tuples $\mathbf{a},\mathbf{b}$ from $F$,
if $F\models(\hat{\eta}_n)_{\mfp,\forall}^\tau(\mathbf{a},\mathbf{b})$ then
$F_\mfP\models\eta_n(\mathbf{a},\mathbf{b})$ for every 
$\mfP\in\mcS_\mfp^\tau(F)$ by Proposition~\ref{dense} and Proposition~\ref{quantifylocal}.
Therefore, if 
$F\models \psi_n(\mathbf{a},\mathbf{b})\wedge\bigwedge_{\mfp\in S} (\hat{\eta}_n)_{\mfp,\forall}^\tau(\mathbf{a},\mathbf{b})$,
then the conditions 
\begin{eqnarray}\label{eqnfn}
 f_n(\mathbf{a},\mathbf{Z})=0,\quad f_n(\mathbf{b},\mathbf{Z})\neq 0
\end{eqnarray}
define a non-singular $F$-variety $V$ which has an $F_\mfP$-rational point
for every $\mfP\in\mcS_S^\tau(F)$.
Thus, since $F$ is $\PStauCL$, $V$ has an $F$-rational point,
so $F\models\eta_n(\mathbf{a},\mathbf{b})$.
Consequently, $F$ satisfies (2).

Conversely, assume that $F$ satisfies $T_{\PStauCC}$.
Let $V$ be any smooth $F$-variety
that has an $F^\prime$-rational point for every $F^\prime\in\CC_S^\tau(F)$.
Since $F'(V)=F'(V')$ for any open subvariety $V'$ of $V$,
Lemma~\ref{varieties} implies that
the $F^\prime$-rational points are Zariski-dense on $V$.
So since $V$ is birationally equivalent to a hypersurface,
we can assume without loss of generality that $V$ is given by
tuples $\mathbf{a}$ resp.~$\mathbf{b}$ from $F$ as in 
(\ref{eqnfn}).
Thus, $F^\prime\models\eta_n(\mathbf{a},\mathbf{b})$ for every $F^\prime\in\CC_S^\tau(F)$,
so $F\models(\hat{\eta}_n)_{\mfp,\forall}^\tau(\mathbf{a},\mathbf{b})$
by (1) and Proposition~\ref{quantifylocal}.
Since $F$ satisfies (2), $F\models\eta_n(\mathbf{a},\mathbf{b})$,
i.e.~$V$ has an $F$-rational point, and so $F$ is $\PStauCC$.
\end{proof}

\begin{Remark}
Clearly, we just proved Theorem \ref{thm1} of the introduction.
Note that Proposition~\ref{PsefCCAx} gives an $\mcLr$-axi\-omat\-i\-za\-tion of 
$\PpC$, $\PRC$, and ${\rm PC}_M$ fields,
cf.~Remark \ref{PCM}.
By Proposition \ref{PCCPSCC}
we also get an $\mcL_{\rm ring}$-axiomatization of the class 
of $\PCC$ fields in the Hilbert-type infinitary logic $L_{\omega_1,\omega}$.
Note that the class of $\PCC$ fields is not elementary in our standard
finitary logic $L_{\omega,\omega}$.

We can use our results to prove
the conjecture posed in
\cite[Remark 11]{DarnierePAC}:
Darni\`ere calls a field $F$ {\bf RC-local} if it is ${\rm P}\mathcal{F}{\rm C}$
for $\mathcal{F}=\CC(F)$, and {\bf restricted RC-local}
if every elementary extension of $F$ satisfies the same property.
Let $\mathcal{F}$ be a finite family of fields taken among $\mathbb{R}$
and the finite extensions of the fields $\hat{\mathbb{Q}}_p$,
and denote by $\mathbf{Q}_{\mathcal{F}}$ the maximal Galois extension of $\mathbb{Q}$
contained in every $F\in\mathcal{F}$.
Then $\mathbf{Q}_{\mathcal{F}}$ is $\PCC$
by \cite{MoretBailly} and \cite{GPR}.
Darni\`ere conjectures that it is restricted RC-local and that
$R_{\mathcal{F}}$,
the intersection over all $p$-valuation rings, is $\mcLr$-definable in $\mathbf{Q}_{\mathcal{F}}$.
The first part of this conjecture follows
from our axiomatization of $\PCC$ fields,
the second part follows from
Proposition~\ref{MainTheoremT}.
\end{Remark}

\begin{Corollary}\label{quantifyPSCC}
Let $\mfp\in S$
and let $\varphi(\mathbf{x})$ be an $\mcLr$-formula.
The $\mcLr(K)$-formula $\hat{\varphi}_{\mfp,\forall}^\tau(\mathbf{x})$ 
of Proposition~\ref{quantifylocal} satisfies the following:
For every $\PStauCC$ field $F\supseteq K$
and for all elements $a_1,\dots,a_m\in F$
the following are equivalent:
\begin{enumerate}
 \item $F\models\hat{\varphi}_{\mfp,\forall}^\tau(\mathbf{a})$.
 \item $F^\prime\models\varphi(\mathbf{a})$ for all $F^\prime\in\CC_\mfp^\tau(F)$.
\end{enumerate}
\end{Corollary}

\begin{proof}
This follows by combining
Proposition~\ref{PsefCCAx}, 
Proposition~\ref{dense},
and Proposition~\ref{quantifylocal}.
\end{proof}

\section{The Strong Approximation Property}
\label{sec6}

\noindent
We prove that the space of orderings of a  $\PStauCC$ field satisfies the
so called
\textquoteleft strong approximation property\textquoteright~of \cite{Prestel1},
first studied in \cite{KnebuschRosenbergWare}.
We need the strong approximation property
for the characterization of totally $S^\tau$-adic extensions in terms of holomorphy domains,
which follows in the next section.

\begin{Definition}\label{DefSAP}
Let $\tilde{\mcS}(F)$ be the set of {\em all} primes of $F$,
and let $\tilde{\mcS}_\mfp(F)$ be the subset of those lying above $\mfpinS$.
We equip $\tilde{\mcS}(F)$ with the following {\bf Zariski-topology}:
A subbasis of open sets is given by sets of the form
$$
 H(a)=\{\mfP\in\tilde{\mcS}(F) \colon a\in\mathcal{O}_\mfP \},
$$
where $a\in F$.
A set $\mcS\subseteq\tilde{\mathcal{S}}(F)$ is
{\bf profinite} if $\mcS$, as a subspace of $\tilde{\mathcal{S}}(F)$, is a profinite space,
i.e.~a totally disconnected compact Hausdorff space.
We say that $\mcS$ satisfies {\bf SAP} (the Strong Approximation Property)
if $\mcS$ is profinite and the family $H(a)\cap\mcS$, $a\in F$, is closed under finite intersections.

Let $\tilde{\mcS}_\mcP(F)=\tilde{\mcS}(F)\setminus\tilde{\mcS}_\infty(F)$ be the set
of non-archimedean primes of $F$.
We also consider
the following (finer) {\bf patch topology} on $\tilde{\mcS}_\mcP(F)$:
A subbasis of open-closed sets is given by sets of the form
$$
 H_\mcP(a) = \{\mfP\in\tilde{\mcS}_\mcP(F) \colon v_\mfP(a)\geq 0 \}
$$
and
$$
 H_\mcP^\prime(a) = \{\mfP\in\tilde{\mcS}_\mcP(F) \colon v_\mfP(a)>0 \},
$$
where $a\in F$.

We say that $F$ is {\bf $S^\tau$-SAP}
if $\mcS_\mfp^\tau(F)$ satisfies {\rm SAP} for each $\mfpinS$.
\end{Definition}

\begin{Lemma}\label{patchclosed}
Let $\mfp\in S$ with $p_\mfp\neq\infty$.
The following subsets of $\tilde{\mcS}_\mfp(F)$ are closed in the patch topology:
\begin{enumerate}
%\item $\tilde{\mcS}_\mfp(F)$
\item $\mcS_{1,e'}:=\{\mfP\in\tilde{\mcS}_\mfp(F) \colon v_\mfP\mbox{ is discrete and }v_\mfP(\pi_\mfp)\leq e'\}$, $e'\in\mathbb{N}$
\item $\mcS_{2,f'}:=\{\mfP\in\tilde{\mcS}_\mfp(F) \colon f'\notdivides f_\mfP\}$, $f'\in\mathbb{N}$
\item $\mcS_\mfp^\tau(F)$
\end{enumerate}
\end{Lemma}

\begin{proof}
(1): For $\mfP\in\tilde{\mcS}_\mfp(F)$,
$v_\mfP(\pi_\mfp)\leq e'$ if and only if for all $a\in F^\times$,
$v_\mfP(a)\leq 0$ or $v_\mfP(a^{e'})\geq v_\mfP(\pi_\mfp)$,
i.e.
$$
 \mcS_{1,e'} = \tilde{\mcS}_\mfp(F)\cap\bigcap_{a\in F^\times} (H_\mcP(a^{-1})\cup H_\mcP(\pi_\mfp^{-1}a^{e'})).
$$
(2): The following are equivalent: $f'|f_\mfP$; $\Phi_{p^{f'}-1}$ has a zero in $\bar{F}_\mfP$;
  there exists $a\in F^\times$ with $v_\mfP(a) \geq 0$ and $v_\mfP(\Phi_{p^{f'}-1}(a))>0$.
 Thus, 
$$
 \mcS_{2,f'} = \tilde{\mcS}_\mfp(F)\cap\bigcap_{a\in F^\times} (H_\mcP'(a^{-1})\cup H_\mcP(\Phi_{p^{f'}-1}(a)^{-1})).
$$
(3): This follows from (1) and (2), since
$$
 \mcS_\mfp^\tau(F)=\mcS_{1,e}\cap\bigcap_{f'\notdivides ff_\mfp}\mcS_{2,f'}.
$$
\end{proof}

\begin{Lemma}\label{boolean}
For every $\mfp\in S$,
$\mcS_\mfp^\tau(F)$ is profinite, and
the family $H(a)\cap\mcS_\mfp^\tau(F)$, $a\in F$, is closed under complements $($in $\mcS_\mfp^\tau(F)$$)$.
\end{Lemma}

\begin{proof}
%\FIRSTCASE{Part A: Case $\ch(\mfp)\neq\infty$}
First, assume that $p_\mfp\neq\infty$.
By \cite[Corollary A.7]{KuhlmannPlaces},
$\tilde{\mcS}_\mfp(F)$ is quasi-compact and hence profinite in the patch topology.
Thus, since $\mcS_\mfp^\tau(F)$ is closed in $\tilde{\mcS}_\mfp(F)$
by Lemma \ref{patchclosed}(3),
also $\mcS_\mfp^\tau(F)$ is profinite in the patch topology.
Lemma~\ref{Hp}(2) implies that
the family $H_\mcP(a)\cap\mcS_\mfp^\tau(F)$, $a\in F$, is closed under complements
and
the patch topology on $\mcS_\mfp^\tau(F)$ coincides with the Zariski-topology,
which proves the claim.

Now assume that $p_\mfp=\infty$.
Since for $a\in F^\times$, $\mcS_\mfp(F)\setminus H(a)=H(-a)\cap\mcS_\mfp(F)$,
the family $H(a)\cap\mcS_\mfp(F)$, $a\in F$, is closed under complements.
By \cite[6.5]{Prestel1}, the Zariski-topology on the space $\mcS_\infty(F)$ of orderings is profinite.
Since
$\mcS_\mfp(F)=\bigcap_{a\in\mcO_\mfp}H(a)\cap\mcS_\infty(F)$
is closed in $\mcS_\infty(F)$,
the claim follows.
\end{proof}

\begin{Lemma}\label{charpSAP}
If $p_\mfp\neq\infty$,
then $\mcS_\mfp^\tau(F)$ satisfies {\rm SAP}.
\end{Lemma}

\begin{proof}
This follows from
Lemma~\ref{boolean} and
Lemma~\ref{Hp}(1).
\end{proof}

\begin{Lemma}\label{algSAP}
If $F/K$ is algebraic, then $F$ is $S^\tau$-{\rm SAP}.
\end{Lemma}

\begin{proof}
Let $\mfp\in S$ and $a,b\in F$. Since $F/K$ is algebraic,
there exists a finite subextension $L/K$ of $F/K$
such that $a,b\in L$.
The weak approximation theorem 
applied to the finite set of local primes (i.e.~absolute values)
$\mcS_\mfp^\tau(L)$ yields $c\in L$ with
$H(a)\cap H(b)\cap\mcS_\mfp^\tau(F)=H(c)\cap\mcS_\mfp^\tau(F)$.
Thus, $\mcS_\mfp^\tau(F)$ satisfies {\rm SAP}.
\end{proof}

If $F$ is $\PRC$, then $\mcS_\infty(F)$ satisfies {\rm SAP}, see \cite[Proposition 1.3]{PrestelPRC}.
In fact this holds for every $\PStauCC$ field.
We prove this by combining the construction of Section \ref{sec2} with
the specific polynomial from \cite{PrestelPRC}.

\begin{Lemma}\label{Lemmaf}
For $a,b\in F^\times$, let
$$
 f_{a,b}(X,Y) = abX^2Y^2 +aX^2+bY^2-1 \in F[X,Y].
$$
If $\mfp\in S$ with $p_\mfp=\infty$, and $\mfP\in\mcS_\mfp(F)$,
then the following holds:
\begin{enumerate}
 \item  $f_{a,b}$ has a zero in $F_\mfP$.
 \item If $x,y\in F$ and $f_{a,b}(x,y)>_\mfP-1$, then 
          $ab(ax^2+by^2)\geq_\mfP0$ if and only if $a\geq_\mfP0$ and $b\geq_\mfP0$.
\end{enumerate}
\end{Lemma}

\begin{proof}
(1): First note that 
$$
 f_{a,b}(X,Y)=aX^2(bY^2+1)+(bY^2-1).
$$
One can choose $y\in F$ such that
$(-\frac{1}{a})\frac{by^2-1}{by^2+1}>_\mfP 0$.
Indeed, 
if $a>_\mfP0$, let $y=0$.
If $a<_\mfP0$ and $b>_\mfP0$, let $y=1+b^{-1}$.
If $a<_\mfP0$ and $b<_\mfP0$, let $y=1-b^{-1}$.
Since $F_\mfP$ is real closed, there exists $x\in F_\mfP$ such that
$x^2=(-\frac{1}{a})\frac{by^2-1}{by^2+1}$,
hence $f_{a,b}(x,y)=0$.

(2):
First note that $f_{a,b}(0,0)=-1$, so $x\neq 0$ or $y\neq 0$.
Furthermore, $f_{a,b}(x,y)>_\mfP-1$ implies that
\begin{eqnarray}\label{eqnabxy}
 ax^2+by^2 >_\mfP -abx^2y^2.
\end{eqnarray}
If $a>_\mfP 0$ and $b>_\mfP 0$, then $ab(ax^2+by^2)\geq_\mfP 0$.
If $a<_\mfP0$ and $b<_\mfP0$, then
$ab>_\mfP0$ and $ax^2+by^2<_\mfP0$ (since $x\neq 0$ or $y\neq 0$),
so $ab(ax^2+by^2)<_\mfP0$.
If $a>_\mfP0$ and $b<_\mfP0$, or $a<_\mfP0$ and $b>_\mfP0$, then
$ab<_\mfP0$ and thus
$ab(ax^2+by^2) <_\mfP -a^2b^2x^2y^2 \leq_\mfP 0$ by (\ref{eqnabxy}).
\end{proof}

\begin{Proposition}\label{PSCCSAP}
If $F$ is $\PStauCC$, then $F$ is $S^\tau$-{\rm SAP}.
\end{Proposition}

\begin{proof}
Let $\mfpinS$. 
If $p_\mfp\neq\infty$, then $S_\mfp^\tau(F)$ satisfies SAP by Lemma~\ref{charpSAP}.
Therefore, assume that $p_\mfp=\infty$, and let $a,b\in F^\times$.
We want to use the polynomials constructed in Section \ref{sec2}.
Recall Lemma~\ref{TimpliesS}, which gives a translation 
from our current setting to Setting~\ref{settingS}.
In particular, write $S=\{\mfp_1,\dots,\mfp_n\}$.
Assume $\mfp=\mfp_i$, and let 
$$
 G_{a,b}(X,Y,Z)= H_2(Z)(1-A_i(X)C(X)f_{a,b}(X,Y))-H_2(1),
$$
where $A_i,C,H_u$ are 
the corresponding polynomials defined in Section \ref{sec2},
and $f_{a,b}$ is as in Lemma \ref{Lemmaf}.
By $\axA$, $\axC$, and Lemma~\ref{Lemmaf}(1),
$A_i(X)C(X)f_{a,b}(X,Y)$ has a zero in $F_\mfP$ for each $\mfP\in\mcS_S^\tau(F)$.
Since $F$ is $\PStauCC$, $\axHHH$ and Lemma~\ref{nonsingular} imply that
there exist $x,y,z\in F$ such that $G_{a,b}(x,y,z)=0$.
Thus, if $\mfP\in\mcS_\mfp(F)$, then
$$
 1-A_i(x)C(x)f_{a,b}(x,y)=\frac{H_2(1)}{H_2(z)}\leq_\mfP \frac{5}{4},
$$
by $\axH$ and $\axHH$,
so $A_i(x)C(x)f_{a,b}(x,y)\geq_\mfP -1/4$.
Since $A_i(x)C(x)>_\mfP1$ by $\axAAAAA$ and $\axCC$,
this implies that 
$f_{a,b}(x,y)\geq_\mfP -1/4>_\mfP-1$.
Therefore, by Lemma~\ref{Lemmaf}(2), 
$H(a)\cap H(b)\cap\mcS_\mfp(F)=H(c)\cap\mcS_\mfp(F)$,
where 
$c=ab(ax^2+by^2)\in F$.
Hence,
$\mcS_\mfp(F)$ satisfies SAP,
as claimed.
\end{proof}

This proves Theorem \ref{thm3} of the introduction.
As Ido Efrat pointed out to me, there might be an alternative approach to Proposition \ref{PSCCSAP}
by deducing the SAP property from Galois theoretic properties of $\PStauCC$ fields, like in the real case in \cite{HaranInvolutions}.

\section{Totally $S^\tau$-adic Field Extensions}
\label{sec7}

\noindent
We conclude this work by defining totally $S^\tau$-adic field extensions and
describing them in terms of holomorphy domains.
This also gives an equivalent definition of the $\PStauCC$ property.

\begin{Definition}
Let $\mfpinS$ and $\tau'\leq\tau$. If $M/F$ is an extension,
let ${\rm res}_\mfp^{\tau'}\colon\mcS_\mfp^{\tau'}(M)\rightarrow\mcS_\mfp^{\tau'}(F)$
be the restriction map 
given by $\mfQ\mapsto\mfQ|_F$.
We call an extension $M/F$ {\bf totally $S^\tau$-adic}
if the restriction map ${\rm res}_\mfp^{\tau'}\colon\mcS_\mfp^{\tau'}(M)\rightarrow\mcS_\mfp^{\tau'}(F)$
is surjective for each $\mfp\in S$ and $\tau'\leq\tau$.
\end{Definition}

\begin{Remark}
Note that
the restriction map 
${\rm res}_\mfp^{\tau'}\colon\mcS_\mfp^{\tau'}(M)\rightarrow\mcS_\mfp^{\tau'}(F)$
is continuous in the Zariski-topology,
and that $M/F$ is totally $S^\tau$-adic if and only if
each $\mfP\in\mcS_S^\tau(F)$ extends to a prime of $M$ of the same type.

If $K=\mathbb{Q}$, $|S|=1$, and $\tau=(1,1)$ then our notion of totally $S^\tau$-adic extensions
coincides with the classical notions of 
totally real extensions (as in \cite{PrestelPRC}, \cite{Ershovtotallyreal})
resp.~totally $p$-adic extensions (as in \cite{Grob}, \cite{Jarden1991}).
The following lemmas unify results from these works.
\end{Remark}

\begin{Lemma}\label{totShom}
The field 
$F$ is $\PStauCC$ if and only if
for every domain $R=F[x_1,\dots,x_n]$ 
which is finitely generated over $F$ and
whose quotient field $M$ is regular and totally $S^\tau$-adic over $F$,
there exists an $F$-homomorphism $R\rightarrow F$.
\end{Lemma}

\begin{proof}
First assume that $F$ is $\PStauCC$.
If $M/F$ is regular,
then $R$ is the coordinate ring of an affine $F$-variety $V$.
Let $\mfp\in S$ and $\mfP\in\mcS_\mfp^\tau(F)$.
By Proposition~\ref{dense},
since $F$ is $\PStauCC$, $\mfP$ is quasi-local.
If $M/F$ is totally $S$-adic,
there exists $\mfQ\in\mcS_\mfp^\tau(M)$ with $\mfQ|_F=\mfP$ and $\tp(\mfQ)=\tp(\mfP)$.
Then $M_\mfQ\supseteq F_\mfP(V)$,
so $V$ has a smooth $F_\mfP$-rational point by
Lemma~\ref{varieties}.
So since $F$ is $\PStauCC$, $V$ has an $F$-rational point,
and therefore there exists an $F$-homomorphism $R\rightarrow F$.

Conversely, let $V$ be a smooth $F$-variety
that has an $F^\prime$-rational point for every $F^\prime\in\CC_S^\tau(F)$.
By Lemma \ref{varieties} we can assume without loss of generality that $V$ is affine.
Then the coordinate ring $R=F[V]$ is a domain which is finitely generated over $F$ and whose quotient field $M=F(V)$ is regular over $F$.
Let $\mfp\in S$, $\mfP\in\mcS_\mfp^\tau(F)$, and $F^\prime\in\CC(F,\mfP)$.
There exists $\mfQ\in\mcS_\mfp^{\tau}(F^\prime(V))$
with $\mfQ|_F=\mfP$ and $\tp(\mfQ)=\tp(\mfP)$
by Lemma~\ref{varieties},
so $\mfQ|_M\in\mcS_\mfp^\tau(M)$, $(\mfQ|_M)|_F=\mfP$, and $\tp(\mfQ|_M)=\tp(\mfP)$.
Hence, $M/F$ is totally $S^\tau$-adic,
so by assumption there exists an $F$-homomorphism $R\rightarrow F$,
i.e.~$V$ has an $F$-rational point, as claimed.
\end{proof}

\begin{Lemma}\label{totallySadic}
Let $\mfp\in S$ and $\tau'\leq\tau$.
If $M/F$ is an extension and $\mcS_\mfp^{\tau'}(F)$ satisfies {\rm SAP},
then the following statements are equivalent:
\begin{enumerate}
 \item ${\rm res}_\mfp^{\tau'}\colon\mcS_\mfp^{\tau'}(M)\rightarrow\mcS_\mfp^{\tau'}(F)$ is surjective.
 \item $R_\mfp^{\tau'}(M)\cap F=R_\mfp^{\tau'}(F)$.
 \item $R_\mfp^{\tau'}(M)\cap F\subseteq R_\mfp^{\tau'}(F)$.
\end{enumerate}
\end{Lemma}

\begin{proof}
$(1)\Rightarrow(2)$:
Assume that ${\rm res}_\mfp^{\tau'}$ is surjective.
Then
$$
 R_\mfp^{\tau'}(F)=\bigcap_{\mfP\in\mcS_\mfp^{\tau'}(F)}\mcO_\mfP=\bigcap_{\mfQ\in\mcS_\mfp^{\tau'}(M)}(\mcO_\mfQ\cap F)=R_\mfp^{\tau'}(M)\cap F.
$$
$(2)\Rightarrow(3)$: This is trivial.

$(3)\Rightarrow(1)$:
Assume that  
${\rm res}_\mfp^{\tau'}$
is not surjective.
By Lemma~\ref{boolean}, $\mcS_\mfp^{\tau'}(M)$ and $\mcS_\mfp^{\tau'}(F)$ are profinite spaces.
Hence,
since ${\rm res}_\mfp^{\tau'}$ is continuous,
${\rm res}_\mfp^{\tau'}(\mcS_\mfp(M))$ is closed in $\mcS_\mfp^{\tau'}(F)$.
Therefore, 
$\mcS_\mfp^{\tau'}(F)\setminus{\rm res}_\mfp^{\tau'}(\mcS_\mfp^{\tau'}(M))$ is nonempty and open.
It follows that the complement of a basic open-closed set contained in
$\mcS_\mfp^{\tau'}(F)\setminus{\rm res}_\mfp^{\tau'}(\mcS_\mfp^{\tau'}(M))$
is an open-closed proper subset $X$ of $\mcS_\mfp^{\tau'}(F)$ containing ${\rm res}_\mfp^{\tau'}(\mcS_\mfp^{\tau'}(M))$.
By Lemma~\ref{boolean},
the subbasis $H(a)\cap\mcS_\mfp^{\tau'}(F)$, $a\in F$, 
of $\mcS_\mfp^{\tau'}(F)$
is closed under complements.
Hence,
since $\mcS_\mfp^{\tau'}(F)$ satisfies {\rm SAP}, $X=H(x)\cap\mcS_\mfp^{\tau'}(F)$ for some $x\in F$ by 
\cite[6.6]{Prestel1}.
Therefore,
$$
 {\rm res}_\mfp^{\tau'}(\mcS_\mfp^{\tau'}(M))\subseteq H(x)\cap\mcS_\mfp^{\tau'}(F)\subsetneqq \mcS_\mfp^{\tau'}(F).
$$
Then $x\in R_\mfp^{\tau'}(M)\cap F$ but $x\notin R_\mfp^{\tau'}(F)$, so 
$R_\mfp^{\tau'}(M)\cap F \not\subseteq R_\mfp^{\tau'}(F)$.
\end{proof}

\begin{Corollary}\label{elementary}
Assume that $F$ is $\PStauCC$.
If $F\prec M$ is an elementary extension,
then $M/F$ is regular and totally $S^\tau$-adic. 
\end{Corollary}

\begin{proof}
Every elementary extension is regular, see for example \cite[7.3.3]{FJ3}.
By Proposition~\ref{PsefCCAx}, since $F$ is $\PStauCC$ and $M\equiv F$, $M$ is $\PStauCC$.
Thus, by Proposition~\ref{MainTheoremT}, since $F\prec M$, $R_\mfp^{\tau'}(M)\cap F=R_\mfp^{\tau'}(F)$ for each $\mfpinS$ and $\tau'\leq\tau$.
By Proposition~\ref{PSCCSAP},
since $F$ is $\PStauCC$,
$F$ is $S^\tau$-{\rm SAP},
so $\mcS_\mfp^{\tau'}(F)$ satisfies SAP for each $\mfp\in S$ and $\tau'\leq\tau$.
Therefore, by Lemma~\ref{totallySadic}, $M/F$ is totally $S^\tau$-adic.
\end{proof}

This finishes the proof of Corollary \ref{cor1} of the introduction.

\end{document}